\documentclass[final,1p,times]{elsarticle}

\usepackage{amssymb}
\usepackage{amsthm}
\usepackage{hyperref}
\usepackage{amsmath,amssymb,amsopn,amsfonts,mathrsfs,amsbsy,amscd}
\usepackage{longtable}
\usepackage{color}
\usepackage{mathtools}

\journal{Journal}
\usepackage{graphicx}
\usepackage{derivative}
\usepackage[all]{xy}

\usepackage{tikz-cd}
\usepackage{stmaryrd}
\usepackage{nccmath}

\newcommand{\G}{{\mathfrak{g}}}

\newcommand{\h}{{\mathfrak{h}}}
\newcommand{\ad}{{\mathrm{ad}}}

\newcommand{\Ad}{{\mathrm{Ad}}}

\font\bb=msbm10

\def\R{\hbox{\bb R}}

\newtheorem{theorem}{Theorem}[section]
\newtheorem{lemma}[theorem]{Lemma}
\newtheorem{definition}[theorem]{Definition}
\newtheorem{example}[theorem]{Example}

\newtheorem{proposition}[theorem]{Proposition}
\newtheorem{remark}[theorem]{Remark}
\newtheorem{corollary}[theorem]{Corollary}
\numberwithin{equation}{section}
\newcommand{\dreqno}{\let\veqno\eqno}

\newcounter{cases}
\newcounter{subcases}[cases]

\begin{document}	
	\begin{frontmatter}
		
		\title{Invariant Poisson Structures on Homogeneous Manifolds: Algebraic Characterization, Symplectic Foliation, and Contravariant Connections\vskip 0.5cm \today}
		
		\author[label1]{Abdelhak Abouqateb}
		\address[label1]{Cadi Ayyad University, UCA, Faculty of Sciences and Technologies, Department of Mathematics B.P.549 Gueliz Marrkesh Morocco}
		\ead{a.abouqateb@uca.ac.ma}
		
		\author[label2]{Charif Bourzik}
		\address[label2]{Higher Institute of Nursing and Health Techniques Guelmim Morocco}
		\ead{bourzikcharif@gmail.com}
		
		
		\begin{abstract} 
	 In this paper, we study invariant Poisson structures on homogeneous manifolds, which serve as a natural generalization of homogeneous symplectic manifolds previously explored in the literature. Our work  begins by providing an algebraic characterization of invariant Poisson structures on homogeneous manifolds. More precisely, we establish a connection between these structures and solutions to a specific type of classical Yang-Baxter equation. This leads us to explain a bijective correspondence between invariant Poisson tensors and class of Lie subalgebras: For a connected Lie group  $G$  with lie algebra $\mathfrak{g}$, and $H$ a connected closed subgroup with Lie algebra $\h$, we demonstrate that the class of $G$-invariant Poisson tensors on $G/H$ is in bijective correspondence with the class of Lie subalgebras $\mathfrak{a} \subset \mathfrak{g}$ containing $\mathfrak{h}$, equipped with a $2$-cocycle $\omega$ satisfying $\mathrm{Rad}(\omega) = \mathfrak{h}$. Then, we explore numerous examples of invariant Poisson structures, focusing on reductive and symmetric pairs. Furthermore, we show that the symplectic foliation associated with invariant Poisson structures consists of homogeneous symplectic manifolds. Finally, we investigate invariant contravariant connections on homogeneous spaces endowed with invariant Poisson structures. This analysis extends the study by K. Nomizu of invariant covariant connections on homogeneous spaces.\\
	\noindent{\bf Mathematics Subject Classification 2020:} 
	 17B45, 22F30, 53D17, 70G45.\\
	  {\bf Keywords:}  Lie groups, Homogeneous spaces, Poisson manifolds, Poisson connections.
		\end{abstract}
	
	\end{frontmatter}
	\section{Introduction}
The study of invariant symplectic structures on homogeneous manifolds has motivated numerous significant contributions in the literature; see, for instance, \cite{Baues, CHU, Boothby,Guan}. Poisson manifolds, which constitute a larger class than the symplectic manifolds, play a fundamental role in modern geometry and have been extensively developed in various works \cite{ VAISMAN, Dufour, Camille}. Given this context, it is natural to investigate invariant Poisson structures on homogeneous manifolds, which is the primary focus of this paper. More precisely, we provide a precise algebraic characterization of invariant Poisson structures on homogeneous manifolds and demonstrate that the symplectic leaves of an invariant Poisson structure are themselves homogeneous symplectic manifolds. Additionally, we explore in this context a precise description of invariant contravariant connections which are an important tool in Poisson geometry as explained in \cite {FERNANDES}. The motivation for this work stems from the study of invariant Koszul-Vinberg structures \cite{abb}, which share many similarities with Poisson structures despite arising in distinct contexts. While Poisson structures concern antisymmetric contravariant tensors, Koszul-Vinberg structures involve symmetric contravariant tensors on affine manifolds. This duality highlights the rich interaction between these two frameworks and underlines the importance of understanding these structures on homogeneous manifolds.
 
A Poisson structure on a manifold $M$ is a skew-symmetric bivector field $\pi\in\Gamma(\wedge^2 TM)$ satisfying, $[\pi,\pi]_{S}=0$, where $[\cdot,\cdot]_{S}$ is the Schouten-Nijenhuis bracket, such bivector field induce a Lie bracket on $\Omega^1(M)$ 
	\begin{equation*}
		[\alpha,\beta]_{\pi}=\mathcal{L}_{\alpha^{\#}}\beta-\mathcal{L}_{\beta^{\#}}\alpha-d\pi(\alpha,\beta),
	\end{equation*}
where $\alpha^{\#}$ is the vector field on $M$ given by $\langle\beta,\alpha^{\#}\rangle=\pi(\alpha,\beta)$.  
	
We recall that a $G$-homogeneous space $M$ is a manifold on which a Lie group $G$ acts smoothly and transitively, in a such case, there exists a closed subgroup $H$ of $G$ such that the quotient $G/H$ is identified with $M$ through a $G$-equivariant diffeomorphism (the subgroup $H$ is the isotropy subgroup at a point $o\in M$). For any $g\in G$, we will denote by  $\lambda_g:G/H\to G/H$ the diffeomorphism given by $\lambda_g(aH)=gaH$ for $a\in G$.
A {\it{$G$-invariant Poisson structure}} on a $G$-homogeneous manifold $G/H$ is a Poisson structure $\pi$ on $G/H$  such that the action of $G$ on $G/H$ preserves the bivector field $\pi$. This means that  for any $g \in G$ and one forms $\alpha, \beta \in \Omega^{1}(G/H)$,
$$
\pi\left( \lambda_{g}^*
\alpha,\lambda_{g}^*
\beta\right)=\pi(\alpha, \beta) \circ \lambda_{g} .$$
In other words, if we denote by $\pi_{\#}:T^*M\rightarrow TM,\alpha\mapsto \alpha^{\#}$ the bundle map given by $\pi_{\#}(\alpha)=\alpha^{\#}$, then the bivector field $\pi$ is $G$-invariant if and only if for any $g \in G,$   $$\pi_{\#, \bar{g}}=T_{\bar{e}}\left(\lambda_{g}\right) \circ \pi_{\#, \bar{e}} \circ\left(T_{\bar{e}}\left(\lambda_{g}\right)\right)^{*},$$
where $\bar{g}=gH$ and $\bar{e}=H$.
This implies that $\pi$ is a regular Poisson structure, i.e. all associated symplectic leaves have the same dimension.

\noindent 
	
	  We will now elucidate the connection between  this work and the concept of {\it {Poisson homogeneous  spaces}} as studied in the existing literature. Indeed,  recall that a  {\it {Poisson Lie}} $(G,\pi_G)$ is a Lie group endowed with a Poisson structure $\pi_G$ such that the multiplication map $G\times G\rightarrow G$ is a Poisson map, where $G\times G$ is equipped with the product Poisson structures. An action of a Poisson Lie group $(G,\pi_G)$ on a Poisson manifold $(M,\pi)$  is called a {\it {Poisson action}} if the action map $G \times M \rightarrow M$ is a Poisson map, where $G \times M$ is equipped with the product Poisson structure. A  {\it {Poisson homogeneous space}} is a Poisson manifold $M$ on which acts transitively a Poisson Lie group $G$ such that the action $G \times M \rightarrow M$ is Poisson action. Such structures have been extensively studied in \cite{Lu,Diatta,Drinfeld}. In this paper, we focus on a specific class of Poisson homogeneous spaces, those where $\pi_G$ is the trivial Poisson structure. Surprisingly, to our knowledge, a comprehensive study of such structures has not yet been undertaken in the literature.  One motivation for investigating these structures lies in their potential to serve as a natural generalization of symplectic homogeneous spaces, as hinted at earlier in this discussion.

	We now provide an overview of the main structural sections of this paper.

	Section \ref{Section2} is devoted to reviewing some basic facts about homogeneous spaces by specifying the notations that will be used throughout the paper.
	
	 In Section \ref{Section3}, we investigate $G$-invariant Poisson bivector fields on homogeneous manifold of the form $G/H$. Specifically, we establish a one-to-one correspondence between such Poisson structures and what we will refer to (cf. Definition \ref{matrices}) as $r$-matrices of the pair $(\G,\mathfrak{h})$ where $\G$ (resp. $\mathfrak{h}$) is the Lie algebra of $G$ (resp. of $H$). Indeed, if $\pi$ is a $G$-invariant skew-symmetric bivector field on $G/H$ and if $r\in\wedge^{2}(\mathfrak{g}/\mathfrak{h})$ is  its associated $\Ad(H)$-invariant bivector,
then we will associate a tensor $\llbracket r,r\rrbracket\in\otimes^3(\mathfrak{g}/\mathfrak{h})$ such that $	\llbracket r,r\rrbracket =0$ if and only if $\pi$ is a Poisson tensor on $G/H$ (cf. Theorem \ref{Poisson theorem}). This leads us to show that there exists a bijective correspondence between invariant Poisson tensors on $G/H$ and class of Lie subalgebras: For a connected Lie group  $G$  with lie algebra $\mathfrak{g}$, and $H$ a connected closed subgroup with Lie algebra $\h$, we demonstrate that the class of $G$-invariant Poisson tensors on $G/H$ is in bijective correspondence with the class of Lie subalgebras $\mathfrak{a} \subset \mathfrak{g}$ containing $\mathfrak{h}$, equipped with a $2$-cocycle $\omega$ satisfying $\mathrm{Rad}(\omega) = \mathfrak{h}$ (cf. Corollary \ref{description}).\\
Of particular interest is the case where $G/H$ is reductive space, i.e., there exists a vector space decomposition $\mathfrak{g}=\mathfrak{h}\oplus\mathfrak{m}$ such that $Ad(H)(\mathfrak{m})=\mathfrak{m}$. In this setting, we provide a one-to-one correspondence between
 $G$-invariant Poisson bivector fields on $G/H$ and a class of $\Ad(H)$-invariant skew-symmetric bivectors $r \in \wedge^{2} \mathfrak{m}$ (Theorem \ref{Poisson reductive}). Notably, when $(G,H)$ is symmetric pair, 
we show that this class is precisely the vector space $\left(\wedge^{2}\mathfrak{m}\right)^{\mathrm{Ad}(H)}$. As applications, we present twos illustrative examples.  In the first example, we prove that there does not exist any non-trivial $\mathrm{GL}_{n}^{+}(\mathbb{R})$-invariant Poisson structure on  the space $\mathcal{S}^{++}_{n}(\mathbb{R})$ of real symmetric positive definite $n\times n$-matrices in the case when $n=4k$, or $n=4k+3$ (Theorem \ref{S_n}).  In the second example (\ref{Grassman manifold}), we compute all $\mathrm{SO}_{4}(\mathbb{R})$-invariant Poisson structures on the oriented Grassmann manifold $\mathrm{G}^{+}_{2}(\mathbb{R}^4)$.

 Moreover, inspired by Nomizu's theorem on invariant affine connections \cite{Nomizu}, in Section \ref{Section4} we investigate $G$-invariant contravariant connections (in the meaning of \cite{FERNANDES}).
More precisely, we were interested in giving an algebraic characterization of such structure on homogeneous manifold  with an invariant Poisson structure. Our main result in this part will be Theorem \ref{ILCC} where we give a one-to-one correspondence between $G$-invariant contravariant connections on $(G/H,\pi)$ and  $\mathrm{Ad}(H)$-invariant bilinear maps $\mathfrak{b}:\mathfrak{m}^*\times\mathfrak{m}^* \rightarrow\mathfrak{m}^*$. 
This process is a powerful algebraic tool for constructing contravariant connections on homogeneous manifolds with an invariant Poisson structure.

  In Section \ref{Section5}, we give a geometric description of the regular symplectic foliation induced by an invariant Poisson structure on $G/H$. We show that the symplectic leaves are homogeneous symplectic manifolds in the sense of \cite{CHU}.

\section{Preliminaries and Notations}\label{Section2}
 We will adopt the same notations and terminology as \cite{abb} but in order to have a self contained document we specify the main tools. Throughout this paper, $G$ will be a connected Lie group with Lie algebra $\mathfrak{g}$, and $H$ a closed subgroup of $G$ with Lie algebra $\mathfrak{h}$, $M:=G / H$. Denote  by 
 $$p: G \rightarrow M, g\mapsto p(g)=\bar{g}:=g H\ ;\;\; q:\mathfrak{g}\rightarrow\mathfrak{g}/\mathfrak{h},\; u\mapsto u+\mathfrak{h}$$ the canonical projections and $\bar{e}:=H$. The action of $G$ on $M$ is defined as follows,
\begin{equation}\label{HB}
	\lambda: G \times M \rightarrow M, \quad\left(g, \overline{g^{\prime}}\right) \mapsto g \cdot \overline{g^{\prime}}=\lambda_{g}\left(\overline{g^{\prime}}\right)=\overline{g g^{\prime}} 
\end{equation}
Furthermore, the tangent linear map $T_{e} p: \mathfrak{g} \rightarrow T_{\bar{e}} M$ is surjective, inducing a linear isomorphism
\begin{equation}
	\Phi_{e}: \mathfrak{g}/\mathfrak{h} \xrightarrow{\cong} T_{\bar{e}} M, \quad \Phi_{e}(u+\mathfrak{h})=T_{e} p(u). 
\end{equation}
Generally, for any $g \in G$,
\begin{equation}
	\Phi_{g}: \mathfrak{g}/\mathfrak{h} \xrightarrow{\cong} T_{\bar{g}} M, \quad \Phi_{g}(u+\mathfrak{h})=T_{\bar{e}} \lambda_{g} \circ T_{e} p(u). 
\end{equation}
This leads to the bundle isomorphism 
\[
\Phi: G \times_H {\mathfrak{g}/\mathfrak{h}}  \xrightarrow{\cong} TM, \quad[g, u+\mathfrak{h}] \mapsto \Phi_{g}(u+\mathfrak{h}),
\]
where $G \times_H {\mathfrak{g}/\mathfrak{h}}$ is the orbit space of $G \times \mathfrak{g}/\mathfrak{h}$ under the right action of $H$ given by $(g, u+\mathfrak{h}) \cdot a=\left(g a, \operatorname{Ad}_{a^{-1}}(u)+\mathfrak{h}\right)$. In other words, the tangent bundle $TM$ is identified with the vector bundle associated to the principal bundle $p:G\to G/H$ and the linear representation
$\overline{\Ad}:H\rightarrow\mathrm{GL}(\mathfrak{g}/\mathfrak{h})$ given by
$$
\overline{\Ad}_a(u+\mathfrak{h})=\Ad_a(u)+\mathfrak{h},\quad \text{for}\;\; a\in H.
$$
The associated dual representation   $H\rightarrow\mathrm{GL}((\mathfrak{g}/\mathfrak{h})^*)$ is defined by
$
 a\cdot \alpha= \overline{\Ad}_{a^{-1}}^{*}\alpha$ for $\alpha\in (\mathfrak{g}/\mathfrak{h})^*$, and its derivative representation  $\mathfrak{h}\rightarrow\mathrm{End}((\mathfrak{g}/\mathfrak{h})^*)$ is given by
  	$
  u\cdot \alpha=- \overline{\mathrm{ad}}_{u}^{*} \alpha$, for $u\in \mathfrak{h}$, i.e., $\langle u\cdot \alpha,v+\mathfrak{h} \rangle=-\langle \alpha,[u,v]+\mathfrak{h} \rangle$ for any $v\in \G$.
\subsection*{Invariant skew-symmetric bivector fields}
Let $\pi$ be a skew-symmetric bivector field on $M$ and denote by $\pi_{\#}:T^*M\rightarrow TM,\alpha\mapsto \alpha^{\#}$ the bundle map given by  $\langle\beta,\alpha^{\#}\rangle:=\pi(\alpha,\beta)$. 
Then it is easy to see that the following conditions are equivalent:
\begin{enumerate}
	\item $\pi$ is $G$-invariant, i.e., for any $g \in G$ and one forms $\alpha, \beta \in \Omega^{1}(M)$,
	$$
	\pi\left( \lambda_{g}^*
	\alpha,\lambda_{g}^*
	\beta\right)=\pi(\alpha, \beta) \circ \lambda_{g} .$$
	\item For any $g \in G$,   $\pi_{\#, \bar{g}}=T_{\bar{e}}\left(\lambda_{g}\right) \circ \pi_{\#, \bar{e}} \circ\left(T_{\bar{e}}\left(\lambda_{g}\right)\right)^{*}$.
	\item For any $u \in \mathfrak{g}$, $\mathcal{L}_{u^{*}} \pi=0$. \\ Here
	$u^*\in \Gamma(TM)$ is the fundamental vector field induced by $\exp(-tu)$, which is a vector field on $M$, associated to $u\in \mathfrak{g}$.
\end{enumerate}
Recall that there is a one-to-one correspondence between $G$-invariant skew-symmetric bivector fields on $M$ and $\Ad(H)$-invariant skew-symmetric bivector $r \in \wedge^{2}(\mathfrak{g}/\mathfrak{h})$. Here the $\Ad(H)$-invariance of the bivector $r$ means that one of the following equivalent conditions is satisfied:
\begin{enumerate}
	\item For any $\alpha,\beta\in(\mathfrak{g}/\mathfrak{h})^{*}$ and $a \in H,$
	\begin{equation}\label{Carac1}
		r(\overline{\mathrm{Ad}}_{a}^{*}\alpha,\overline{\mathrm{Ad}}_{a}^{*}\beta)=
		r(\alpha,\beta),
	\end{equation}
or equivalently
\begin{equation}\label{Carac2}
	(\overline{\mathrm{Ad}}_{a}^{*}\alpha)^{\#}=	\overline{\mathrm{Ad}}_{a^{-1}}(\alpha^{\#}),
\end{equation}where $\alpha^{\#}\in\mathfrak{g}/\mathfrak{h}$  is given by $\langle \beta,\alpha^{\#}\rangle:=r(\alpha,\beta).$
	\item If $H$ is connected, for any $\alpha,\beta\in(\mathfrak{g}/\mathfrak{h})^{*}$ and $u \in \mathfrak{h}$,
	\begin{equation}\label{Cara3}
		r(\overline{\mathrm{ad}}_{u}^{*}\alpha, \beta)+r( \alpha,\overline{\mathrm{ad}}_{u}^{*}\beta)=0, 
	\end{equation} 
or equivalently,
    	\begin{equation}\label{Cara4}
    	(\overline{\mathrm{ad}}_{u}^{*}\alpha)^{\#}=-\overline{\mathrm{ad}}_{u}(\alpha^{\#}). 
    \end{equation} 
\end{enumerate}

\section{Invariant Poisson structures on $G/H$}\label{Section3}

Let $\pi$ be a $G$-invariant skew-symmetric bivector field on $G/H$ and $r\in(\wedge^{2}(\mathfrak{g}/\mathfrak{h}))^{\Ad(H)}$ its associated $\Ad(H)$-invariant bivector. Let $\widetilde{r}\in\wedge^{2}\mathfrak{g}$ be any bivector satisfying $\wedge^{2}q (\widetilde{r})=r$, this means that for any $\bar{\eta},\bar{\xi}\in(\mathfrak{g}/\mathfrak{h})^*$, we have $\widetilde{r}(q^*\bar{\eta},q^*\bar{\xi})=r(\bar{\eta},\bar{\xi})$. \\ For any $\eta\in\mathfrak{g}$ denote by $\eta^{\widetilde{\#}}$ the vector in $\mathfrak{g}$ given by $\langle \xi, \eta^{\widetilde{\#}}\rangle:=\widetilde{r}(\eta,\xi),\,\forall\xi\in\mathfrak{g}^*$.  In what follows,  $\mathfrak{h}^{\circ}$ will be the annihilator of $\mathfrak{h}$ in $\mathfrak{g}^*$, i.e., the set of linear functionals on $\mathfrak{g}$ which vanish on $\mathfrak{h}$. A canonical isomorphism from $(\mathfrak{g}/\mathfrak{h})^*$ to $\mathfrak{h}^{\circ}$  is given by $\alpha\mapsto q^*(\alpha)$.

\begin{lemma}
	For any $a\in H$ and $\eta,\xi\in\mathfrak{h}^{\circ}$ we have 
	\begin{eqnarray}
		\widetilde{r}(\mathrm{Ad}_{a}^{*}\eta,\mathrm{Ad}_{a}^{*}\xi)=\widetilde{r}(\eta,\xi),\label{Ad invariance 1}\\
		(\mathrm{Ad}_{a}^{*}\eta)^{\widetilde{\#}}-\mathrm{Ad}_{a^{-1}}(\eta^{\widetilde{\#}})\in\mathfrak{h},\label{Ad invariance 2}
	\end{eqnarray}
and for any $u\in\mathfrak{h}$,
\begin{eqnarray}
\widetilde{r}(\mathrm{ad}_{u}^{*}\eta,\xi)+\widetilde{r}(\eta,\mathrm{ad}_{u}^{*}\xi)=0,\label{ad invariance 1}\\	(\mathrm{ad}_{u}^{*}\eta)^{\widetilde{\#}}+\mathrm{ad}_{u}(\eta^{\widetilde{\#}})\in\mathfrak{h}.\label{ad invariance 2}
\end{eqnarray}
\end{lemma}
\begin{proof}
	Let $a\in H$, $\eta,\xi\in\mathfrak{h}^{\circ}$ and $\bar{\eta},\bar{\xi}\in(\mathfrak{g}/\mathfrak{h})^*$ such that $q^*\bar{\eta}=\eta$ and $q^*\bar{\xi}=\xi$. For identity (\ref{Ad invariance 1}), just write: 
	\begin{align*}
	\displaystyle	\widetilde{r}(\mathrm{Ad}_{a}^{*}\eta,\mathrm{Ad}_{a}^{*}\xi)
		&=	\widetilde{r}(\mathrm{Ad}_{a}^{*}q^*\bar{\eta},\mathrm{Ad}_{a}^{*}q^*\bar{\xi})\\
		&=\widetilde{r}(q^*\overline{\mathrm{Ad}}_{a}^{*}\bar{\eta},q^*\overline{\mathrm{Ad}}_{a}\bar{\xi})\\
		&=r(\overline{\mathrm{Ad}}_{a}^{*}\bar{\eta},\overline{\mathrm{Ad}}_{a}^{*}\bar{\xi})\\
		&\stackrel{(\ref{Carac1})}{=}r(\bar{\eta},\bar{\xi})
		\\
		&=\widetilde{r}(\eta,\xi).
	\end{align*}
For  (\ref{Ad invariance 2}), it follows from the following:
	\begin{align*}
		q((\mathrm{Ad}_{a}^{*}\eta)^{\widetilde{\#}})&=	q((\mathrm{Ad}_{a}^{*}q^*\bar{\eta})^{\widetilde{\#}})\\
		&=q((q^*\overline{\mathrm{Ad}}_{a}^{*}\bar{\eta})^{\widetilde{\#}})\\
		&=(\overline{\mathrm{Ad}}_{a}^{*}\bar{\eta})^{\#}\\
		&\stackrel{(\ref{Carac2})}{=}\overline{\mathrm{Ad}}_{a^{-1}}(\bar{\eta}^{\#})\\
		&= \overline{\mathrm{Ad}}_{a^{-1}}(q(\eta^{\widetilde{\#}}))\\
		&=q(\mathrm{Ad}_{a^{-1}}(\eta^{\widetilde{\#}})).
	\end{align*}
Equations (\ref{ad invariance 1}) and (\ref{ad invariance 2}) are a direct consequences of (\ref{Ad invariance 1}) and (\ref{Ad invariance 2}) respectively.
\end{proof}

Define a bracket $[\cdot,\cdot]_{\widetilde{r}}$ on $\mathfrak{g}^{*}$  by setting 
\begin{equation}\label{Bracket}
	\displaystyle
	[\eta,\xi]_{\widetilde{r}}:=-\ad_{\eta^{\widetilde{\#}}}^{*}\xi+\ad_{\xi^{\widetilde{\#}}}^{*}\eta.
\end{equation} 
\begin{lemma} For any $\eta,\xi\in\mathfrak{h}^{\circ}$, we have
	\begin{equation}\label{rbracket}
	[\eta,\xi]_{\widetilde{r}}\in\mathfrak{h}^{\circ},
	\end{equation} 
\end{lemma}
\begin{proof} Let $\eta,\xi\in\mathfrak{h}^{\circ}$ and  $u\in\mathfrak{h}$,
	\begin{align*}
	\langle [\eta,\xi]_{\widetilde{r}}, u\rangle &= +\langle \xi, [\eta^{\widetilde{\#}},u]\rangle-\langle \eta, [\xi^{\widetilde{\#}},u]\rangle \\
	&\stackrel{(\ref{ad invariance 2})}{=} -\langle \xi, (\ad_{u}^*\eta)^{\widetilde{\#}}\rangle+\langle \eta, (\ad_{u}^*\xi)^{\widetilde{\#}}\rangle\\
	&= -\widetilde{r}(\mathrm{ad}_{u}^{*}\eta,\xi)-\widetilde{r}(\eta,\mathrm{ad}_{u}^{*}\xi)\\
	&\stackrel{(\ref{ad invariance 1})}{=}0.
	\end{align*}
	Hence $[\eta,\xi]_{\widetilde{r}}\in\mathfrak{h}^{\circ}$.
\end{proof}
\begin{proposition}
	The element $\llbracket r,r\rrbracket\in\otimes^3(\mathfrak{g}/\mathfrak{h})$ given by 
	\begin{equation}
		\llbracket r,r\rrbracket(\eta,\xi,\varepsilon):=
		\langle \varepsilon,[\eta,\xi]_{\widetilde{r}}^{\widetilde{\#}}-[\eta^{\widetilde{\#}},\xi^{\widetilde{\#}}] \rangle,\qquad \forall \eta,\xi,\varepsilon \in \mathfrak{h}^{\circ}
	\end{equation}
	does not depend on $\widetilde{r}$.
\end{proposition}
\begin{proof}
 Let $\eta,\xi,\varepsilon\in\mathfrak{h}^{\circ}$ and $\widetilde{r}_{1},\widetilde{r}_{2}\in\wedge^{2}\mathfrak{g}$ such that $\wedge^{2}q(\widetilde{r}_{i})=r$. This allows us to consider
 $
  u:=\eta^{\widetilde{\#}_1}-\eta^{\widetilde{\#}_2}\in\mathfrak{h}\; \text{and}\;\; v:=\xi^{\widetilde{\#}_1}-\xi^{\widetilde{\#}_2}\in\mathfrak{h}.
 $
  Hence we have
	\begin{align*}
		\langle \varepsilon,[\eta,\xi]_{\widetilde{r}_{1}}^{\widetilde{\#}_1}-[\eta,\xi]_{\widetilde{r}_{2}}^{\widetilde{\#}_2}\rangle&
		\stackrel{(\ref{rbracket})}{=}\langle \varepsilon,[\eta,\xi]_{\widetilde{r}_{1}}^{\widetilde{\#}_1}-[\eta,\xi]_{\widetilde{r}_{2}}^{\widetilde{\#}_1}\rangle\\
		&\stackrel{(\ref{Bracket})}{=}\langle\varepsilon,-(\ad_{u}^{*}\xi)^{\widetilde{\#}_1}+(\ad_{v}^{*}\eta)^{\widetilde{\#}_1}\rangle \\
		&=\langle\varepsilon,-(\ad_{u}^{*}\xi)^{\widetilde{\#}_1}+(\ad_{v}^{*}\eta)^{\widetilde{\#}_2}\rangle\\
		&\stackrel{(\ref{ad invariance 2})}{=}\langle\varepsilon,[u,\xi^{\widetilde{\#}_1}]-[v,\eta^{\widetilde{\#}_2}]\rangle\\
		&=\langle\varepsilon,[\eta^{\widetilde{\#}_1},\xi^{\widetilde{\#}_1}]-[\eta^{\widetilde{\#}_2},\xi^{\widetilde{\#}_2}]\rangle.
	\end{align*}
	Which proves that $\llbracket r,r\rrbracket$ depend only on $r$.		
\end{proof}
Our main results in this section can be setting as follows.
\begin{theorem}\label{Poisson theorem}
	$\pi$ is a Poisson bivector field on $G/H$ if and only if
	\begin{equation}\label{Yang Baxter}
		\llbracket r,r\rrbracket=0.
	\end{equation}
\end{theorem}
\begin{proof}
		Since $\pi$ is a $G$-invariant skew-symmetric bivector field,  it follows that $[\pi, \pi]_S$ is also $G$-invariant. Therefore, $[\pi,\pi]_S=0$ if and only if $[\pi,\pi]_{S}(\bar{e})=0$. Denote by $\widetilde{r}^{+}$ the left invariant bivector filed on $G$ associated to $\widetilde{r}$. Then, for any $g\in G$ and $\alpha,\beta\in \Omega^{1}(M)$,
	$$
		\pi(\alpha,\beta)(\bar{g})=r\left(\Phi_{g}^{*} \alpha_{\bar{g}}, \Phi_{g}^{*} \beta_{\bar{g}}\right)
	=\widetilde{r}\left(q^*\Phi_{g}^{*} \alpha_{\bar{g}}, q^*\Phi_{g}^{*} \beta_{\bar{g}}\right)
	=\widetilde{r}\left({L}_g^*p^{*} \alpha_{\bar{g}}, {L}_g^*p^{*} \beta_{\bar{g}}\right),
	$$
	where $L_g$ is the left translation by $g$. Hence, $
	\pi(\alpha,\beta)(\bar{g})=\widetilde{r}^{+}\left( p^{*}\alpha,p^{*} \beta\right)(g)$. This shows that $\pi$ and $\widetilde{r}^{+}$ are $p$-related, and then $[\pi,\pi]_S$ and $[\widetilde{r}^{+}, \widetilde{r}^{+}]_S$ are also $p$-related. From this fact, we obtain
	$$
		[\pi, \pi]_S (\bar{e}) = \wedge (T_{e}p)\left([\widetilde{r}^{+}, \widetilde{r}^{+}]_S (\bar{e}) \right)
	= \wedge \Phi_{e} \left(\wedge q(\llbracket \widetilde{r}, \widetilde{r} \rrbracket_{\mathrm{AS}})\right),
	$$
	where $\llbracket \cdot, \cdot \rrbracket_{\mathrm{AS}}$ is the algebraic Schouten bracket on the Lie algebra $\mathfrak{g}$. \\ Now, since $\wedge \Phi_{e}$ is an isomorphism  it follows that $[\pi, \pi]_S = 0$ if and only if $\wedge q(\llbracket \widetilde{r}, \widetilde{r} \rrbracket_{\mathrm{AS}})=0$, which is equivalent to say that for any $\eta,\xi,\varepsilon\in\mathfrak{h}^{\circ}$,
	$\llbracket  \widetilde{r},  \widetilde{r} \rrbracket_{\mathrm{AS}}(\eta,\xi,\varepsilon)=0$. \\ Finally, the desired equivalence is deduced from the following fact
	\begin{align*}
		\displaystyle\frac{1}{2}\llbracket \widetilde{r}, \widetilde{r} \rrbracket_{\mathrm{AS}}(\eta,\xi,\varepsilon)&=-\langle \eta,[\xi^{\widetilde{\#}},\varepsilon^{\widetilde{\#}}]\rangle-\langle \xi,[\varepsilon^{\widetilde{\#}},\eta^{\widetilde{\#}}]\rangle-\langle \varepsilon,[\eta^{\widetilde{\#}},\xi^{\widetilde{\#}}]\rangle\\
		&=\langle \varepsilon,[\eta,\xi]_{\widetilde{r}}^{\widetilde{\#}}-[\eta^{\widetilde{\#}},\xi^{\widetilde{\#}}]\rangle\\
		&=\llbracket r,r\rrbracket(\eta,\xi,\varepsilon).
	\end{align*}
\end{proof}

As a first direct consequence of Theorem \ref{Poisson theorem} we get the well known correspondence in the case of Lie groups. More precisely, if $G$ is a connected Lie group, then there is a one-to-one correspondence between left invariant Poisson bivector fields on $G$ and bivectors $r \in \wedge^{2}\mathfrak{g}$ satisfying the classical Yang-Baxter equation: 	$[\eta,\xi]_{r}^{\#} = [\eta^{\#},\xi^{\#}]$, 	where $[\eta,\xi]_{r}:=\ad_{\eta^{\#}}^{*}\xi-\ad_{\xi^{\#}}^{*}\eta$. The solutions of such equations are called $r$-matrices (one can see \cite{Alioune} for a nice description). This motivates the following definition.
\begin{definition}\label{matrices} Let $r\in(\wedge^{2}(\mathfrak{g}/\mathfrak{h}))^{\Ad(H)}$, $\llbracket r,r\rrbracket=0$ is called equation of Yang-Baxter type. And  solutions of these equations are called $r$-matrices.  
\end{definition}

As we know in the cases of Lie algebra, if $r$ is an $r$-matrix in $\mathfrak{g}$ then $(\mathfrak{g}^*,[\cdot,\cdot]_r)$ is a Lie algebra and the linear map $r_{\#}:(\mathfrak{g}^*,[\cdot,\cdot]_r)\rightarrow (\mathfrak{g},[\cdot,\cdot])$ is a morphism of Lie algebras. Now we are going to generalize this property to the case of $r$-matrices associated to a pair $(\mathfrak{g},\mathfrak{h})$. In what follows $(\mathfrak{g}/\mathfrak{h})^H$  (resp. $(\mathfrak{h}^{\circ})^H$ ) will be the vector subspace of $\Ad(H)$-invariant elements of $\mathfrak{g}/\mathfrak{h}$ (resp. $\Ad^*(H)$-invariant elements of $\mathfrak{h}^{\circ}$), and $r\in(\wedge^{2}(\mathfrak{g}/\mathfrak{h}))^{\Ad(H)}$.
\begin{lemma}$\ $
\begin{enumerate}
		\item The vector space $(\mathfrak{g}/\mathfrak{h})^H$
	endowed with the bracket
		$[u+\mathfrak{h},v+\mathfrak{h}]_{\mathfrak{g}/\mathfrak{h}}:=[u,v]+\mathfrak{h}$,
	is a Lie algebra.
	\item For any $\eta,\xi\in(\mathfrak{h}^{\circ})^H$, we have $ [\eta,\xi]_{\widetilde{r}}\in(\mathfrak{h}^{\circ})^H$,  and the bracket $[\cdot,\cdot]_{r}=[\cdot,\cdot]_{\widetilde{r}}$ defined in $(\mathfrak{h}^{\circ})^H$ does not depend on $\widetilde{r}$.
\end{enumerate}
\end{lemma}
\begin{proof}
	\begin{enumerate}
			\item At first, we will show that the bracket \([\cdot,\cdot]_{\mathfrak{g}/\mathfrak{h}}\) on the space  $(\mathfrak{g}/\mathfrak{h})^H$ is well-defined and \(\text{Ad}(H)\)-invariant. The proof proceeds in two steps.\\
\underline{Step 1}: The bracket is well-defined.
Let \(u + \mathfrak{h}, v + \mathfrak{h} \in (\mathfrak{g}/\mathfrak{h})^H\), and let \(x, y \in \mathfrak{h}\). We need to show that
$
[u + x, v + y] - [u, v] \in \mathfrak{h}$.

1. Since \(u + \mathfrak{h}, v + \mathfrak{h} \in (\mathfrak{g}/\mathfrak{h})^H\), for any \(t \in \R\), we have:
\[
\text{Ad}_{\exp(ty)}(u) - u \in \mathfrak{h} \quad \text{and} \quad \text{Ad}_{\exp(tx)}(v) - v \in \mathfrak{h}.
\]
This implies that $
[y, u] \in \mathfrak{h} \quad \text{and} \quad [x, v] \in \mathfrak{h}$.

2. Expanding:
$
[u + x, v + y] = [u, v] + [u, y] + [x, v] + [x, y]$.
Since \([u, y], [x, v], [x, y] \in \mathfrak{h}\), it follows that: $
[u + x, v + y] - [u, v] \in \mathfrak{h}$.
Thus, the bracket is well-defined on $(\mathfrak{g}/\mathfrak{h})^H$.
\\
\underline{Step 2}: The bracket is \(\text{Ad}(H)\)-invariant.
Let \(u + \mathfrak{h}, v + \mathfrak{h} \in (\mathfrak{g}/\mathfrak{h})^H\). We are going to show that $[u+\mathfrak{h},v+\mathfrak{h}]_{\mathfrak{g}/\mathfrak{h}}$ is $\Ad(H)$-invariant. To do this, consider $a\in H$, then by definition, there exist \(x, y \in \mathfrak{h}\) such that:
$
\text{Ad}_a(u) = u + x$ and $\text{Ad}_a(v) = v + y$.
1. Applying \(\text{Ad}_a\) to the bracket \([u, v] + \mathfrak{h}\):
$
\overline{\text{Ad}_a}([u, v] + \mathfrak{h}) = \text{Ad}_a([u, v]) + \mathfrak{h}$.

2. Using the fact that \(\text{Ad}_a\) is a Lie algebra homomorphism:
\[
\text{Ad}_a([u, v]) = [\text{Ad}_a(u), \text{Ad}_a(v)] = [u + x, v + y].
\]

3. By expanding as before, we obtain: $
\overline{\text{Ad}_a}([u, v] + \mathfrak{h}) = [u, v] + \mathfrak{h}$.\\
Conclusion: The bracket \([\cdot,\cdot]_{\mathfrak{g}/\mathfrak{h}}\) is well-defined and \(\text{Ad}(H)\)-invariant. The Jacobi identity for \([\cdot,\cdot]_{\mathfrak{g}/\mathfrak{h}}\) follows directly from the Jacobi identity for the Lie bracket \([\cdot,\cdot]\) on \(\mathfrak{g}\). 
This completes the proof.
\item Let $\eta,\xi\in(\mathfrak{h}^{\circ})^H$ and $a\in H$. From $(\ref{Ad invariance 2})$ there exists $x\in\mathfrak{h}$ such that 
		\begin{equation*}
			\Ad_{a^{-1}}(\eta^{\widetilde{\#}})=(\Ad_{a}^{*}\eta)^{\widetilde{\#}}+x=\eta^{\widetilde{\#}}+x.
		\end{equation*}
	So for any $u\in\mathfrak{g}$,
	\begin{align*}
		\langle \Ad_{a}^{*}(\ad_{\eta^{\widetilde{\#}}}^{*}\xi),u\rangle &=\langle \xi,[\eta^{\widetilde{\#}},\Ad_{a}(u)] \rangle\\
		&=\langle \xi,\Ad_{a}([\Ad_{a^{-1}}(\eta^{\widetilde{\#}}),u]) \rangle\\
		&=\langle \xi,[\eta^{\widetilde{\#}},u]+[x,u] \rangle\\
		&=\langle \ad_{\eta^{\widetilde{\#}}}^{*}\xi,u \rangle.
	\end{align*}
		This implies that $[\eta,\xi]_{\widetilde{r}}=-\ad_{\eta^{\widetilde{\#}}}^{*}\xi+\ad_{\xi^{\widetilde{\#}}}^{*}\eta\in(\mathfrak{h}^{\circ})^H$.
		 
		Let $\widetilde{r}_{1},\widetilde{r}_{2}\in\wedge^{2}\mathfrak{g}$ such that $(\wedge^2q)\widetilde{r}_{i}=r$. We set $u:=\eta^{\widetilde{\#}_1}-\eta^{\widetilde{\#}_2}\in\mathfrak{h}$ and $v:=\xi^{\widetilde{\#}_1}-\xi^{\widetilde{\#}_2}\in\mathfrak{h}$. So we have
			$[\eta,\xi]_{\widetilde{r}_1}-[\eta,\xi]_{\widetilde{r}_2}=-\ad_{u}^{*}\xi+\ad_{v}^{*}\eta=0$,
		  hence $[\cdot,\cdot]_{r}$ depend only on $r$.
	\end{enumerate}
\end{proof}

If we identify the vector space $(\mathfrak{g}/\mathfrak{h})^{*}$ with $\mathfrak{h}^{\circ}$ then the linear map $r_{\#}:(\mathfrak{g}/\mathfrak{h})^{*}\rightarrow\mathfrak{g}/\mathfrak{h}$ can be seen as map from $\mathfrak{h}^{\circ}$ to $\mathfrak{g}/\mathfrak{h}$.
\begin{proposition}\label{morphism}
 Let $r\in(\wedge^{2}(\mathfrak{g}/\mathfrak{h}))^{\Ad(H)}$ be a solution of (\ref{Yang Baxter}). Then $((\mathfrak{h}^\circ)^H, [\cdot,\cdot]_{r})$ is a Lie algebra and $r_{\#}: ((\mathfrak{h}^\circ)^H, [\cdot,\cdot]_{r})\rightarrow ((\mathfrak{g}/\mathfrak{h})^H,[\cdot,\cdot]_{\mathfrak{g}/\mathfrak{h}})$ is a Lie algebras morphism.
\end{proposition}
\begin{proof}
Let $\widetilde{r}\in\wedge^2\mathfrak{g}$ satisfying $(\wedge^2q)\widetilde{r}=r$. Let $\eta,\xi,\varepsilon\in(\mathfrak{h}^{\circ})^H$, so from (\ref{Yang Baxter}) there exists $u_0\in\mathfrak{h}$ such that $	[\eta,\xi]_{\widetilde{r}}^{\widetilde{\#}}=[\eta^{\widetilde{\#}},\xi^{\widetilde{\#}}]+u_0$. Hence
 	$\ad_{[\eta,\xi]_{\widetilde{r}}^{\widetilde{\#}}}^{*}\varepsilon=\ad_{[\eta^{\widetilde{\#}},\xi^{\widetilde{\#}}]}^{*}\varepsilon$. Then for any $u\in\mathfrak{g}$,
\begin{align*}
	\langle [[\eta,\xi]_{\widetilde{r}},\varepsilon]_{\widetilde{r}},u\rangle&= \langle -\ad_{[\eta,\xi]_{\widetilde{r}}^{\widetilde{\#}}}^{*}\varepsilon+\ad_{\varepsilon^{\widetilde{\#}}}^{*}[\eta,\xi]_{\widetilde{r}},u\rangle\\
	&=-\langle\varepsilon,[[\eta^{\widetilde{\#}},\xi^{\widetilde{\#}}],u]\rangle-\langle\xi,[\eta^{\widetilde{\#}},[\varepsilon^{\widetilde{\#}},u]]\rangle+\langle\eta,[\xi^{\widetilde{\#}},[\varepsilon^{\widetilde{\#}},u]]\rangle.
\end{align*} 
In the same way we get
\begin{equation*}
	\langle [[\xi,\varepsilon]_{\widetilde{r}},\eta]_{\widetilde{r}},u\rangle=-\langle\eta,[[\xi^{\widetilde{\#}},\varepsilon^{\widetilde{\#}}],u]\rangle-\langle\varepsilon,[\xi^{\widetilde{\#}},[\eta^{\widetilde{\#}},u]]\rangle+\langle\xi,[\varepsilon^{\widetilde{\#}},[\eta^{\widetilde{\#}},u]]\rangle,
\end{equation*}
and
\begin{equation*}
	\langle [[\varepsilon,\eta]_{\widetilde{r}},\xi]_{\widetilde{r}},u\rangle=-\langle\xi,[[\varepsilon^{\widetilde{\#}},\eta^{\widetilde{\#}}],u]\rangle-\langle\eta,[\varepsilon^{\widetilde{\#}},[\xi^{\widetilde{\#}},u]]\rangle+\langle\varepsilon,[\eta^{\widetilde{\#}},[\xi^{\widetilde{\#}},u]]\rangle.
\end{equation*}
Hence we get 
\begin{equation*}
[[\eta,\xi]_{\widetilde{r}},\varepsilon]_{\widetilde{r}}+[[\xi,\varepsilon]_{\widetilde{r}},\eta]_{\widetilde{r}}+[[\varepsilon,\eta]_{\widetilde{r}},\xi]_{\widetilde{r}}=0.
\end{equation*}
Which proves that $[\cdot,\cdot]_{r}$ is a Lie bracket on $(\mathfrak{h}^{\circ})^H$ and $r_{\#}: ((\mathfrak{h}^\circ)^H, [\cdot,\cdot]_{r})\rightarrow ((\mathfrak{g}/\mathfrak{h})^H,[\cdot,\cdot]_{\mathfrak{g}/\mathfrak{h}})$ is a Lie algebras morphism.
\end{proof}

\subsection{Description of $r$-matrices}
Let $M{=}G/H$ be a $G$-homogeneous manifold endowed with a $G$-invariant Poisson structure $\pi$. Denote by $r\in\wedge^{2}(\mathfrak{g}/\mathfrak{h})$ the $\Ad(H)$-invariant bivector associated to $\pi$, and $\widetilde{r}\in\wedge^2\mathfrak{g}$ any bivector satisfying $q\circ \widetilde{r}_{\#}\circ q^*=r_{\#}$. Denote by $\mathfrak{a}_r=q^{-1}(Im(r_{\#}))$.
\begin{proposition}\label{Lie algebra}
	$\mathfrak{a}_r$ is a Lie subalgebra of $\mathfrak{g}$, which contains $\mathfrak{h}$.
\end{proposition}
\begin{proof}
	Let $x,y\in\mathfrak{a}_r$. Then there exists  $\eta,\xi\in\mathfrak{h}^{\circ}$ such that $x=\eta^{\widetilde{\#}}+x_0$ and  $y=\xi^{\widetilde{\#}}+y_0$ where $x_0,y_0\in\mathfrak{h}$. Hence we have
	\begin{align*}
		\displaystyle	q([x,y])&=q([\eta^{\widetilde{\#}},\xi^{\widetilde{\#}}])+q(\ad_{x_0}(\xi^{\widetilde{\#}}))-q(\ad_{y_0}(\eta^{\widetilde{\#}}))\\
		&=q([\eta,\xi]_{\widetilde{r}}^{\widetilde{\#}})+q((\ad_{x_0}^{*}\xi)^{\widetilde{\#}})-q((\ad_{y_0}^{*}\eta)^{\widetilde{\#}})\\
		&=q\left(\left([\eta,\xi]_{\widetilde{r}}+\ad_{x_0}^{*}\xi-\ad_{y_0}^{*}\eta\right)^{\widetilde{\#}}\right).
	\end{align*} 
	Thus $[x,y]\in\mathfrak{a}_r$.
\end{proof}
We define on the Lie algebra $\mathfrak{a}_r$ the following skew-symmetric bilinear form  $$\omega_r:\mathfrak{a}_r\times\mathfrak{a}_r\rightarrow\mathbb{R},\,(x,y)\mapsto\omega_r(x,y)=\widetilde{r}(\eta,\xi),$$ 
where $\eta,\xi\in\mathfrak{h}^{\circ}$ are any elements satisfying $q(\eta^{\widetilde{\#}})=q(x)$ and $q(\xi^{\widetilde{\#}})=q(y)$.   \begin{proposition}\label{2-cocycle}$\ $
	\begin{enumerate}
		\item $\omega_r$ is well defined and depend only on $r$,
		\item $\omega_r$ is $\Ad(H)$-invariant,
		\item $\omega_r$ is $2$-cocycle of $\mathfrak{a}_r$, i.e.,\begin{equation*}
			\omega_r([x,y],z)+\omega_r([y,z],x)+\omega_r([z,x],y)=0,
		\end{equation*} 
		\item $\mathrm{Rad}(\omega_r)=\{x\in\mathfrak{a}_r\,|\, i_{x}\omega_r=0\}=\mathfrak{h}$.
	\end{enumerate}
\end{proposition}
\begin{proof}
	\begin{enumerate}
		\item Let $\eta_1,\eta_2,\xi\in\mathfrak{h}^{\circ}$ such that $q(\eta_{1}^{\widetilde{\#}})=q(\eta_{2}^{\widetilde{\#}})$. Then we have
		\begin{equation*}
			\widetilde{r}(\eta_{1},\xi)-\widetilde{r}(\eta_{2},\xi)=\langle \xi,\eta_{1}^{\widetilde{\#}}-\eta_{2}^{\widetilde{\#}} \rangle
			=0.
		\end{equation*}
		Hence $\omega_r$ is well defined. To see that $\omega_r$ depend only on $r$ we consider $\widetilde{r}_1,\widetilde{r}_2\in\wedge^2\mathfrak{g}$ satisfying $q^*\widetilde{r}_i=r$. Then for any $\eta,\xi\in\mathfrak{h}^\circ$,
		\begin{equation*}
				\widetilde{r}_1(\eta,\xi)-\widetilde{r}_2(\eta,\xi)=\langle \xi,\eta^{\widetilde{\#}_1}-\eta^{\widetilde{\#}_2} \rangle=0.
		\end{equation*}
		\item Follow directly from (\ref{Ad invariance 1}).
		\item Let $x,y\in\mathfrak{a}_r$ from the proof of Proposition \ref{Lie algebra}, there exists $\eta,\xi\in\mathfrak{h}^{\circ}$ and $x_0,y_0\in\mathfrak{h}$ such that 
		\begin{equation*}
			q([x,y])=q\left(\left([\eta,\xi]_{\widetilde{r}}+\ad_{x_0}^{*}\xi-\ad_{y_0}^{*}\eta\right)^{\widetilde{\#}}\right).
		\end{equation*}
		Let $z\in\mathfrak{a}_r$ and $\varepsilon\in\mathfrak{h}^{\circ}$ such that $q(\varepsilon^{\widetilde{\#}})=q(x)$. Hence we have
		\begin{equation*}
			\omega_r([x,y],z)=\widetilde{r}([\eta,\xi]_{\widetilde{r}},\varepsilon)+\widetilde{r}(\ad_{x_0}^{*}\xi,\varepsilon)-\widetilde{r}(\ad_{y_0}^{*}\eta,\varepsilon).
		\end{equation*}
		In the same way there exists $z_0\in\mathfrak{h}$ such that 
		\begin{equation*}
			\omega_r([y,z],x)=\widetilde{r}([\xi,\varepsilon]_{\widetilde{r}},\eta)+\widetilde{r}(\ad_{y_0}^{*}\varepsilon,\eta)-\widetilde{r}(\ad_{z_0}^{*}\xi,\eta),
		\end{equation*}
		and 
		\begin{equation*}
			\omega_r([z,x],y)=\widetilde{r}([\varepsilon,\eta]_{\widetilde{r}},\xi)+\widetilde{r}(\ad_{z_0}^{*}\eta,\xi)-\widetilde{r}(\ad_{x_0}^{*}\varepsilon,\xi).
		\end{equation*}
		Hence from (\ref{Yang Baxter}) and (\ref{ad invariance 1}) it follows that $\omega_r$ is a $2$-cocycle.
		\item Let $x\in\mathrm{Rad}(\omega_r)$ and $\eta\in\mathfrak{h}^{\circ}$ such that $q(\eta^{\widetilde{\#}})=q(x)$. Hence for all $\xi\in\mathfrak{h}^{\circ}$, we have
		\begin{equation*}
			0=\widetilde{r}(\eta,\xi)=\langle\xi,\eta^{\widetilde{\#}}\rangle=\langle\xi,x\rangle.
		\end{equation*}
		Which implies that $x\in\mathfrak{h}$. Hence $\mathrm{Rad}(\omega_r)\subset\mathfrak{h}$ and the other inclusion is obvious.
	\end{enumerate}
\end{proof}
Conversely, given an $\Ad(H)$-invariant Lie subalgebra $\mathfrak{a} \subset \mathfrak{g}$ containing $\mathfrak{h}$, equipped with an $\Ad(H)$-invariant $2$-cocycle $\omega$ satisfying $\mathrm{Rad}(\omega) = \mathfrak{h}$, we can construct an $r$-matrix on $\mathfrak{g}/\mathfrak{h}$ such that the associated Lie algebra $\mathfrak{a}_r$ coincides with $\mathfrak{a}$. More precisely, $r$ is given by the following diagram.
$$\begin{tikzcd}
\left(\mathfrak{a}/\mathfrak{h}\right)^{*}\arrow[r, "(\omega_{\#})^{-1}"] 
	& \mathfrak{a}/\mathfrak{h}  \arrow["\iota",d]  \\
	\arrow[r, "r_{\#}"] \arrow[u, "\iota^*"] \left(\mathfrak{g}/\mathfrak{h}\right)^{*} 
	&  \mathfrak{g}/\mathfrak{h}.  
\end{tikzcd}$$
\begin{proposition} $r$ is an $r$-matrix.
\end{proposition}
\begin{proof}
Similar to the proof of the third assertion of Proposition \ref{2-cocycle}.
\end{proof}
\begin{corollary}\label{description}
	When $H$ is connected, the class of $G$-invariant Poisson tensors on $G/H$ is in bijective correspondence with the class of Lie subalgebras $\mathfrak{a} \subset \mathfrak{g}$ containing $\mathfrak{h}$, equipped with a $2$-cocycle $\omega$ satisfying $\mathrm{Rad}(\omega) = \mathfrak{h}$.
\end{corollary}
\subsection{$r$-matrices in the cases of the reductive pairs and symmetric pairs}
Let $(G,H)$ be a reductive pair with fixed decomposition  $\mathfrak{g}=\mathfrak{h}\oplus\mathfrak{m}$ and  $\mathrm{Ad}_H(\mathfrak{m})=\mathfrak{m}$.  For any $u \in \mathfrak{g}$, $u_{\mathfrak{m}}$ (resp. $u_{\mathfrak{h}}$) denotes the canonical projection of $u$ on $\mathfrak{m}$ (resp. on $\mathfrak{h}$). In this case, we have a canonical identification between $\mathfrak{h}^{\circ}$ and $\mathfrak{m}^*$ the dual vector space of $\mathfrak{m}$, which allows us to transfer the bracket given by (\ref{rbracket}) to a bracket   $\mathfrak{m}^*$; more precisely it is given by
\begin{equation}\label{reductive equation}
	[\alpha,\beta]_{r}=\left(-\mathrm{ad}_{\alpha^{\#}}^{*}\widetilde{\beta}+\mathrm{ad}_{\beta^{\#}}^{*}\widetilde{\alpha}\right)|_{\mathfrak{m}},
\end{equation} 
where $\widetilde{\alpha}$ is the canonical extension of $\alpha$ defined by $\widetilde{\alpha}|_{\mathfrak{m}}=\alpha$ and $\widetilde{\alpha}|_{\mathfrak{h}}=0$. Moreover, one can see easily that the vector space $\mathfrak{m}^H$ endowed with the bracket $(u,v)\mapsto [u,v]_{\mathfrak{m}}$ is a Lie algebra, and the map $u\mapsto u+\mathfrak{h}$ defines a canonical isomorphism with the Lie algebra $((\mathfrak{g}/\mathfrak{h})^H,[\cdot,\cdot]_{\mathfrak{g}/\mathfrak{h}})$. This leads to the following consequence:
\begin{theorem}\label{Poisson reductive}
 Let $(G,H)$ be a reductive pair with fixed decomposition  $\mathfrak{g}=\mathfrak{h}\oplus\mathfrak{m}$ and  $\mathrm{Ad}_H(\mathfrak{m})=\mathfrak{m}$. There is a one to one correspondence between $G$-invariant Poisson bivector fields on $G/H$ and bivectors $r\in(\wedge^{2}\mathfrak{m})^{H}$ satisfying,
 \begin{equation}\label{eq reduc}
 [\alpha,\beta]_{r}^{\#}=[\alpha^{\#},\beta^{\#}]_{\mathfrak{m}},\qquad \forall \alpha,\beta \in \mathfrak{m}^*
 \end{equation} Moreover, if $r$ is a solution of  (\ref{eq reduc}), then $((\mathfrak{m}^*)^H, [\cdot,\cdot]_{r})$ is a Lie algebra and $r_{\#}: ((\mathfrak{m}^*)^H, [\cdot,\cdot]_{r}) \rightarrow (\mathfrak{m}^H, [\cdot,\cdot]_{\mathfrak{m}})$ is a Lie algebras morphism.
\end{theorem}
\begin{proof} Let $\pi$ be a $G$-invariant skew-symmetric bivector field on $G/H$ and $r\in(\wedge^{2}(\mathfrak{g}/\mathfrak{h}))^{\Ad(H)}$ its associated $\Ad(H)$-invariant bivector. Denote by $\iota:\mathfrak{m}\hookrightarrow \G$ the canonical injection and let $I:\mathfrak{m}\to \mathfrak{g}/\mathfrak{h}$ the canonical isomorphism which is $\Ad(H)$-equivariant and is given by $I=q\circ\iota$. Now, the bivector
	$r^\mathfrak{m}\in(\wedge^{2}\mathfrak{m})$ given by
	$
	r^\mathfrak{m}_{\#}=I^{-1}\circ r_{\#}\circ (I^*)^{-1}
	$,
	is $\Ad(H)$-invariant and satisfy \ref{eq reduc}. Indeed, if we consider the natural bivector on $\G$ given by 
		$$
	\widetilde{r}_{\#}=\iota\circ r^\mathfrak{m}_{\#}\circ \iota^*,
	$$ then we have $\wedge^{2}q (\widetilde{r})=r$, and this allows us to prove the formula by a straitforward computation. \\ Conversely, let $r^\mathfrak{m}\in(\wedge^{2}\mathfrak{m})^{H}$ satisfying \ref{eq reduc}, then in the same way the bivector 
	$$r:=I \circ r^\mathfrak{m}_{\#}\circ I^* \in (\wedge^{2}(\mathfrak{g}/\mathfrak{h}))^{\Ad(H)}$$ satisfy 		$\llbracket r,r\rrbracket=0$. 
	
\end{proof}
 \begin{remark} One can naturally wonder if an $r\in \left(\wedge^2(\mathfrak{g}/\mathfrak{h})\right)^H$ satisfy equation (\ref{Yang Baxter}) just for element in $(\mathfrak{h}^\circ)^H$ then it is satisfied for any element of  $\mathfrak{h}^\circ$. Of course this fact in general it is note true, which will be illustrated in the following example. 
 \end{remark}

\begin{example}
	The Poincare group $G=\mathrm{ISO}(2)$ is the group of affine transformations of $\mathbb{R}^2$ which preserve the Lorentz metric. It is isomorphic to the group of $3\times 3$-matrices of the form 
	\[
	\begin{bmatrix}
		A & x \\
		0 & 1
	\end{bmatrix},
\]
where $A\in O(1,1)$ and $x\in\mathbb{R}^2$. The associated Lie algebra $\mathfrak{g}=\mathfrak{iso}(2)$ has as a basis
\[
e_1=\begin{bmatrix}
	0 & 0 & 1 \\
		0 & 0 & -1 \\
			0 & 0 & 0 
\end{bmatrix},\, e_2=\begin{bmatrix}
0 & 0 & 1 \\
0 & 0 & 1 \\
0 & 0 & 0 
\end{bmatrix},\, e_3=\begin{bmatrix}
0 & 1 & 0 \\
1 & 0 & 0 \\
0 & 0 & 0 
\end{bmatrix}.
\]
The structure equations with respect to this basis are 
\[
[e_1,e_2]=0,\, [e_1,e_3]=e_1,\, [e_2,e_3]=-e_2.
\]
Let $\Gamma$ be the discrete subgroup of $G$ given by 
\[
\gamma=\begin{bmatrix}
1 & 0 & n \\
0 & 1 & 0 \\
0 & 0 & 1 
\end{bmatrix},
\]
where $n\in\mathbb{Z}$. The action of $\Ad(\Gamma)$ on the basis $(e_1,e_2,e_3)$ is given by
\[\mathrm{Ad}_{\gamma}(e_1)=e_1,\,\mathrm{Ad}_{\gamma}(e_2)=e_2,\, \mathrm{Ad}_{\gamma}(e_3)=e_3+\frac{n}{2}(e_1-e_2). \]
A direct computation shows that any $\Ad(\Gamma)$-invariant $r$-matrix on $\G$ is of the form $$r=\lambda e_1\wedge e_2,$$
where $\lambda\in\R$. If we take $s=(e_1-e_2)\wedge e_3$ then obviously $s$ is $\Ad(\Gamma)$-invariant and satisfy equation (\ref{reductive equation}) for any element in $(\mathfrak{g}^*)^\Gamma=\R e_{3}^{*}$ but it is not an $r$-matrix.
\end{example}

Recall that if $\Gamma$ is a lattice in $G$ and $G/\Gamma$ admit a $G$-invariant symplectic structure, then $G$ is abelian and $G/\Gamma$ is a torus (see \cite[Theorem 3.7]{Medina}). Hence, if $G$ is non abelian then any $G$-invariant Poisson tensor on $G/\Gamma$ must be degenerate. In the following, we will study an example.
\begin{example}
	Denote by 
	\[ \displaystyle \mathcal{H}_{2n+1}=\left\{\left.{X(a,b,c):=\begin{pmatrix}1&a^t&c\\0&I_n&b\\0&0&1\\\end{pmatrix}}~\right|~a,b\in\mathbb{R}^n,c\in \R\right\} \]
	the $(2n+1)$-dimensional Heisenberg Lie group and by $\Gamma$ the lattice in $\mathcal{H}_{2n+1}$ given by 
	\[\displaystyle \Gamma=\left\{ \gamma(m,q,p):=\left.{\begin{pmatrix}1&m^t&p\\0&I_n&q\\0&0&1\\\end{pmatrix}}~\right|~m,q\in\mathbb{Z}^n,p\in \mathbb{Z}\right\}.
\]
	The Lie algebra of $\mathcal{H}_{2n+1}$ is given by
	\[ \mathfrak{h}_{2n+1}=\{u_1,\ldots,u_n,v_1,\ldots,v_n,w \,|\, [u_i,v_i]=w\}, \]
	where $u_i=X(e_i,0,0),\, v_i=X(0,e_i,0), \, w=X(0,0,1)$
	and $(e_1,\ldots,e_n)$ is the canonical basis of $\mathbb{R}^n$.
	One can check easily that
	\[\left\{
	\begin{array}{lll}
		\mathrm{Ad}_{\gamma(m,q,p)}(u_i)=u_i-(e_{i}^{t}q)w,& \\
		
		\mathrm{Ad}_{\gamma(m,q,p)}(v_i)=v_i+(m^te_{i})w,& \\
		
		\mathrm{Ad}_{\gamma(m,q,p)}(w)=w& 
	\end{array}
	\right. \] 
	Hence 
	\[(\wedge^2\mathfrak{h}_{2n+1})^\Gamma=\left\{\left(\sum_{i=1}^{n}\lambda_i u_i+\mu_i v_i\right)\wedge w \,|\, \lambda_i,\mu_i\in\R \right\}.\]
	Since the only non zero bracket is $[u_i,v_i]=w$ it follows that any element in $(\wedge^2\mathfrak{h}_{2n+1})^\Gamma$ define an $\mathcal{H}_{2n+1}$-invariant Poisson bivector field on the compact Heisenberg manifold $\mathcal{H}_{2n+1}/\Gamma$. 
\end{example} 

Now suppose that $(G,H)$ is a symmetric pair, that is we have a canonical reductive decomposition$ \mathfrak{g}=\mathfrak{h}\oplus\mathfrak{m},\,\mathrm{Ad}_H(\mathfrak{m})=\mathfrak{m}$ such that $[\mathfrak{m},\mathfrak{m}]\subset\mathfrak{h}$. We are going to see that, any element of $(\wedge^{2}\mathfrak{m})^H$ is a solution of the Yang-Baxter equation.
\begin{corollary}\label{CorSym}
 There is a one to one correspondence between $G$-invariant Poisson bivector fields on $G/H$ and bivectors $r\in(\wedge^{2}\mathfrak{m})^H$.
\end{corollary}
\begin{proof} Let $r\in(\wedge^{2}\mathfrak{m})^H$. Since, $[\mathfrak{m},\mathfrak{m}]\subset\mathfrak{h}$ it follows that for any $\alpha,\beta\in\mathfrak{m}^*$,  $[\alpha^{\#},\beta^{\#}]_{\mathfrak{m}}=0$, and from (\ref{reductive equation}) we get that $[\alpha,\beta]_r=0$. 
\end{proof}
\begin{remark}
If \( H \) is a compact and connected Lie group, then we can use the Haar integral, normalized so that \( \int_H da = 1 \), to compute the dimension of \( (\wedge^{2}\mathfrak{m})^H \). More precisely, according to \cite[p. 53]{Greub}, one can consider the function \( f:\mathbb{R}\to\mathbb{R} \) defined by  
\[
f(t)=\int_{H}\det(\rho(a)+t \,\mathrm{Id}_\mathfrak{m})\, da,
\]  
where \( \rho \) denotes the adjoint representation of \( H \) on \( \mathfrak{m} \). Then, the result asserts that
\[
\dim (\wedge^{2}\mathfrak{m})^H = \frac{1}{2} f^{\prime\prime}(0).
\]
\end{remark}

The following subsections illustrate Corollary \ref{CorSym}.
\subsection{$\mathrm{GL}_n^+(\mathbb{R}) $-invariant Poisson structures on $ \mathcal{S}_n^{++}(\mathbb{R})$}
Let $n\geq 2$ be an integer. The space of real symmetric positive definite matrices \( \mathcal{S}_n^{++}(\mathbb{R}) \),  is an open subset of the vector space \( \mathfrak{m} := \mathcal{S}_n(\mathbb{R}) \) of real symmetric \( n \times n \) matrices. We note that the connected Lie group \( G := \mathrm{GL}_n^+(\mathbb{R}) \) acts transitively on \( \mathcal{S}_n^{++}(\mathbb{R}) \) by the action \( g \cdot x = g x g^T \), where \( g \in G \) and \( x \in \mathcal{S}_n^{++}(\mathbb{R}) \). The isotropy subgroup of \( G \) at the identity matrix \( I_n \) is \( H := \mathrm{SO}_n(\mathbb{R}) \). Hence, we have a diffeomorphism \( G/H \cong \mathcal{S}_n^{++}(\mathbb{R}) \), given by \( \bar{g} \mapsto g g^T \). The pair \( (G, H) \) is symmetric, and the canonical decomposition of the Lie algebra \( \mathfrak{g} = \mathfrak{gl}_n(\mathbb{R}) \) is $
\mathfrak{g} = \mathfrak{so}_n(\mathbb{R}) \oplus \mathfrak{m}.$

\begin{theorem}\label{S_n}
There are no non-trivial \( \mathrm{GL}_n^+(\mathbb{R}) \)-invariant Poisson structures on \( \mathcal{S}_n^{++}(\mathbb{R}) \) for \( n = 4k \) or \( n = 4k+3 \).
\end{theorem}
In the following lemma, we will consider the vector subspace $W:=\{ u\in \mathfrak{m} \,|\, \mathrm{tr}(u)=0 \}$, which is clearly invariant
by the adjoint representation $\Ad:H\to  \mathrm{GL}(\mathfrak{m})$.

\begin{lemma}\label{Irred} 
	The induced representation of $H$ on the vector space $W$ is irreducible. 
\end{lemma}
\begin{proof}
Let \( V \subset W \) be a non-trivial \( H \)-invariant vector subspace. Denote by \( D_n(\mathbb{R}) \) the vector space of diagonal matrices. Let \( u_0 \in V \) with \( u_0 \neq 0 \). Then, there exists \( a \in H \) such that \( au_0a^t = d = \mathrm{diag}(\lambda_1,\ldots,\lambda_n) \neq 0 \), $\lambda_1+\cdots+\lambda_n=0$. Without loss of generality, we can suppose that \( \lambda_1 - \lambda_2 \neq 0 \). Now, denote by  $F_{ij}:=E_{ij}-E_{ji}$, for  $1\leq i< j\leq n$ the canonical basis of $ \mathfrak{so}_n(\mathbb{R})$ where $E_{ij}$ is the canonical basis of $\mathfrak{gl}_{n}(\mathbb{R})$. Since $V$ is $H$-invariant, the bracket $v:=[F_{12},d]$ belongs to $V$. A direct computation shows
\[
 v=
\begin{bmatrix}
	0 & \lambda_2 - \lambda_1 & 0 & \cdots & 0 \\
	\lambda_2 - \lambda_1 & 0 & 0 & \cdots & 0 \\
	0 & 0 & 0 & \cdots & 0 \\
	\vdots & \vdots & \vdots & \ddots & \vdots \\
	0 & 0 & 0 & \cdots & 0
\end{bmatrix}.
\]
And it is clear that there exists $P\in  \mathrm{SO}_n(\mathbb{R})$ such that $$P  v P^t = \mathrm{diag}(\lambda_2 - \lambda_1, \lambda_1 - \lambda_2, 0, \dots, 0),$$ which implies that
$
d_2:= \mathrm{diag}(1, -1, 0, \dots, 0)
$ belongs to $V$. 
Now, the canonical basis of the vector space \( D_n(\mathbb{R})\cap W \) is the family $
\{d_2, \dots, d_{n} \}
$,
where each \( d_i \) is defined as: $d_i:=\mathrm{diag}(1, \dots,-1,\dots)$, where the \(-1 \) is in the \( (i, i) \)-position, all other entries are \( 0 \). It is easy to see that for each $i=3,\dots,n$ there exits $q_i\in  \mathrm{SO}_n(\mathbb{R})$ such that $d_i=q_id_2q_i^t$; and since $V$ is $\mathrm{SO}_n(\mathbb{R})$-invariant, we conclude that the vector space $ D_n(\mathbb{R}) \cap W$ is included in $V$. We are then in position to prove the inclusion of $W$ in $V$: Let \( u \in W \). Then, there exists \( Q\in  \mathrm{SO}_n(\mathbb{R}) \) such that  
$
Q u Q^t \in D_n(\mathbb{R}) \cap W $ which is included in $V$; thus $u\in V$. We obtain  \( V = W \).
\end{proof}
We will denote by: 
$$
\prec, \succ : \mathfrak{m} \times \mathfrak{m} \to \mathbb{R},\qquad  (u, v) \mapsto \mathrm{tr}(uv) $$ the canonical scalar product on \( \mathfrak{m} \). Then, the vector space \( (\wedge^2 \mathfrak{m})^H \) is isomorphic to the space of \( H \)-equivariant skew-symmetric endomorphisms of \( \mathfrak{m} \) with respect to \( \prec, \succ \).
\begin{proof}[Proof of Theorem \ref{S_n}]
Let $R:\mathfrak{m} \to \mathfrak{m}$ be an  \( H \)-equivariant skew-symmetric endomorphisms. We have  
\[
 \mathrm{tr}(R(I_n)) =\prec R(I_n), I_n \succ = -\prec I_n, R(I_n) \succ=- \mathrm{tr}(R(I_n))\]
thus \( \mathrm{tr}(R(I_n)) = 0 \), hence \( R(I_n) \in W \). On the other hand, from the equivariance of $R$ we obtain that for any \( a \in H \), we have  
\[
a R(I_n) a^t = R(I_n).
\]  
Thus, according to Lemma \ref{Irred}, we conclude that \( R(I_n) = 0 \).  \\ Now, let \( u \in \mathfrak{m} \). Then, we obtain  
\[
 \mathrm{tr}(R(u)) =\prec R(u), I_n \succ = -\prec u, R(I_n) \succ = 0,
\]  
which implies that \( R(u) \in W \). Therefore, the restriction \( R_{|W}  \) can be viewed as an \( H \)-equivariant and \( \prec, \succ \)-skew symmetric endomorphism of $W$. Hence, by Lemma \ref{Irred}, we have either \( R_{|W} = 0 \) or \( R_{|W} \) is an isomorphism. In the cases when \( n = 4k \) or \( n = 4k+3 \), the dimension of $W$, $\dim W=\frac{n^2+n}{2}-1$ is an odd number, this implies that \( R_{|W} \) cannot be an isomorphism. This proves that $R=0$.
\end{proof}
\begin{remark}
	Let $n=2$ then,
	\[ (\wedge^2\mathfrak{m})^H=\{ \lambda e_1\wedge e_2\,| \,\lambda\in\R \},\]
	where $e_1= \begin{bmatrix}
		1 & 0 \\ 
		0 & -1
	\end{bmatrix} $ and $e_2= \begin{bmatrix}
	0 & 1 \\ 
	1 & 0
\end{bmatrix} $.
\end{remark}
 \subsection{$\mathrm{SO}_{4}(\mathbb{R})$-invariant Poisson structures on $\mathrm{G}^{+}_{2}(\mathbb{R}^4)$}\label{Grassman manifold} 	
 The oriented Grassmann manifold \( \mathrm{G}^{+}_{2}(\mathbb{R}^4) \) is the space of all 
 \( 2 \)-dimensional oriented subspaces of \( \mathbb{R}^4 \).
 We know that \( G = \mathrm{SO}_{4}(\mathbb{R}) \) acts transitively on \( \mathrm{G}^{+}_{2}(\mathbb{R}^4) \). Indeed, if \( V \in \mathrm{G}^{+}_{2}(\mathbb{R}^4) \) and \( g \in G \), then the action of \( g \) on \( V \) is given by \( g \cdot V = g(V) \). The isotropy subgroup of a given oriented subspace \( V \) is isomorphic to \( \mathrm{SO}_{2}(\mathbb{R}) \times \mathrm{SO}_{2}(\mathbb{R}) \). 
 
 This arises because \( \mathrm{SO}_{2}(\mathbb{R}) \) acts on the \( 2 \)-dimensional subspace \( V \), preserving its orientation, and another copy of \( \mathrm{SO}_{2}(\mathbb{R}) \) acts on the complementary \( 2 \)-dimensional subspace \( V^{\perp} \) (the orthogonal complement of \( V \) in \( \mathbb{R}^{4} \) with respect to the canonical scalar product of \( \mathbb{R}^{4} \)), preserving the orientation of this subspace as well. Hence,  $\mathrm{G}^{+}_{2}(\mathbb{R}^4) \cong G/H$, where  $H= 
 \left\{ 
 \begin{bmatrix}
 	a & 0 \\ 
 	0 & b
 \end{bmatrix} 
 \;\middle|\;
 a,b \in \mathrm{SO}_{2}(\mathbb{R})
 \right\}$ is a symmetric space with the canonical decomposition  $\mathfrak{g}=\mathfrak{h}\oplus\mathfrak{m}$,
 where  $\mathfrak{h}= 
 \left\{ 
 \begin{bmatrix}
 	u & 0 \\ 
 	0 & v
 \end{bmatrix} 
 \;\middle|\;
 u,v\in \mathfrak{so}_{2}(\mathbb{R})
 \right\}$,
 and  $\mathfrak{m} = \left\{ 
 \begin{bmatrix}
 	0 &  w\\ 
 	-w^t & 0
 \end{bmatrix}
 \;\middle|\;
 w \in \mathfrak{gl}_{2}(\mathbb{R})
 \right\}$.
 The action of \( H \) on \( \mathfrak{m} \) is given by  $
 \begin{bmatrix}
 	a & 0 \\ 
 	0 & b
 \end{bmatrix} 
 \cdot
 \begin{bmatrix}
 	0 & w \\ 
 	-w^t & 0
 \end{bmatrix}
 =
 \begin{bmatrix}
 	0 & awb^t \\ 
 	-bw^ta^t & 0
 \end{bmatrix}$.
 In what follows, we will describe the vector space \( (\wedge^2\mathfrak{m})^H \). The vector space \( \mathfrak{m} \) has as a basis  $\{  e_1=E_{13}-E_{31}, e_2=E_{23}-E_{32}, e_3=E_{14}-E_{41}, e_4=E_{24}-E_{42}\}$.
 For any \( \theta, t \in\mathbb{R} \), let  
 \[
 a(\theta,t)=\begin{bmatrix}
 	\cos(\theta) & \sin(\theta) & 0 & 0 \\ 
 	-\sin(\theta) & \cos(\theta) & 0 & 0 \\ 
 	0 & 0 & \cos(t) & -\sin(t) \\ 
 	0 & 0 & \sin(t) & \cos(t)
 \end{bmatrix}.
 \]
 The action of \( a(\theta,t) \) on the basis \( (e_1,e_2,e_3,e_4) \) is given by  
 \begin{equation*}
 	\begin{cases}
 		a(\theta,t)\cdot e_1 = \cos(\theta)\cos(t)e_1 - \sin(\theta)\cos(t)e_2 + \cos(\theta)\sin(t)e_3 - \sin(\theta)\sin(t)e_4, & \\  
 		a(\theta,t)\cdot e_2 = \sin(\theta)\cos(t)e_1 + \cos(\theta)\cos(t)e_2 + \sin(\theta)\sin(t)e_3 + \cos(\theta)\sin(t)e_4, & \\  
 		a(\theta,t)\cdot e_3 = -\cos(\theta)\sin(t)e_1 + \sin(\theta)\sin(t)e_2 + \cos(\theta)\cos(t)e_3 - \sin(\theta)\cos(t)e_4, & \\  
 		a(\theta,t)\cdot e_4 = -\sin(\theta)\sin(t)e_1 - \cos(\theta)\sin(t)e_2 + \sin(\theta)\cos(t)e_3 + \cos(\theta)\cos(t)e_4. &  
 	\end{cases} 
 \end{equation*}
 Using these equations, we can prove that  
 \[
 (\wedge^2\mathfrak{m})^H = 
 \left\{ e_1\wedge (\lambda e_2+\mu e_3)+(\mu e_2+\lambda e_3)\wedge e_4
 \;\middle|\;
 \lambda,\mu \in\mathbb{R}
 \right\}.
 \]
 \vskip 0.5cm
 
In the following, let's recap and give more details. If we come back to the initial description of  $r$-matrices, it can be easily shown that  in the case of reductive pairs we have
 \begin{theorem}\label{PoissonReductive2}
 	Let $(G,H)$ be a reductive pair with fixed decomposition  $\mathfrak{g}=\mathfrak{h}\oplus\mathfrak{m}$ and  $\mathrm{Ad}_H(\mathfrak{m})=\mathfrak{m}$. There is a one to one correspondence between $G$-invariant Poisson bivector fields on $G/H$ and the pairs $(W ,\omega )$ where $W$ is an 
 	$\Ad(H)$-invariant vector subspaces $W \subset \mathfrak{m}$ and $\omega$ is an $\Ad(H)$-invariant symplectic form on it such that the two properties hold:
 	\begin{enumerate}
 		\item For all $x,y\in W$, $[x,y]_\mathfrak{m} \in W$,
 		\item for all $x,y,z\in W$,
 		\begin{equation}\label{Cocy}
 		\omega\left([x,y]_\mathfrak{m},z\right)+\omega\left([z,x]_\mathfrak{m},y\right)+\omega\left([y,z]_\mathfrak{m},x\right)=0,\qquad\forall\,x,y,z\in W.
 		\end{equation}
 	\end{enumerate}
 \end{theorem}
As consequences, we have:\begin{itemize}
	\item For  a symmetric pair $(G,H)$, we obtain that there is a one to one correspondence between $G$-invariant Poisson bivector fields on $G/H$ and the pairs $(W ,\omega )$ where $W$ is an $\Ad(H)$-invariant vector subspaces $W \subset \mathfrak{m}$ and $\omega$ is an $\Ad(H)$-invariant symplectic form on it.
	\item 
		If the isotropy representation \( \Ad:H\to \operatorname{GL}(\mathfrak{m}) \) is irreducible, then any \( G \)-invariant Poisson bivector field on \( G/H \) is trivial or non-degenerate (i.e., symplectic). We obtain, in particular, that the only \( \operatorname{SO}(2n+2) \)-invariant Poisson bivector field on the sphere \( S^{2n+1} \) is the trivial one.
		\item For any connected Lie group $G$,  there is a natural symmetric pair $(G\times G,G,\sigma)$, where the involution $\sigma$ is given by $\sigma(a,b)=(b,a) $. The associated symmetric space $G\times G/G$ can be identified with the homogeneous space 
	$M:=G$, where the transitive action of $G\times G$ on $G$  is given by $(a,b)\cdot x:=axb^{-1}$ for $a,b,x\in G$. Applying Corollary \ref{PoissonReductive2} to this setting yields the following characterization: "{ {There is a one to one correspondence between bi-invariant Poisson structures on $G$ and the pairs $(W ,\omega )$ where $W$ is an abelian ideal of $\G$ and $\omega$ is an $\ad(\G)$-invariant symplectic form on it}}" (This resutlt is discussed in\cite{VAISMAN} p. 161).
\end{itemize}
\section{Description of the regular symplectic foliation}\label{Section5}
Let $M{=}G/H$ be a $G$-homogeneous manifold endowed with a $G$-invariant Poisson structure $\pi$. Denote by $r\in\wedge^{2}(\mathfrak{g}/\mathfrak{h})$ the $\Ad(H)$-invariant bivector associated to $\pi$, and $\widetilde{r}\in\wedge^2\mathfrak{g}$ any bivector satisfying $q\circ \widetilde{r}_{\#}\circ q^*=r_{\#}$. Recall from the description of $r$-matrices that $\mathfrak{a}_r=q^{-1}(Im(r_{\#}))$ is a Lie subalgebra of $\mathfrak{g}$ which contains $\mathfrak{h}$. If we denote by $A_r$ the simply connected immersed Lie subgroup of $G$ which integrate the Lie subalgebra $\mathfrak{a}_r$ of $\mathfrak{g}$. Then $2$-cocycle $\omega_r$ defines a homogeneous symplectic structure on $A_r/H$ in the sense of \cite{CHU,Boothby}.

\begin{lemma}\label{invariance}
	For any $b\in H$ we have $c_{b}(A_r)=A_r$,
	where $c_b(a)=bab^{-1}$.
\end{lemma}
\begin{proof}
	Since $A_r$ is connected, then it suffices to show that $\Ad_{b}(\mathfrak{a}_r)=\mathfrak{a}_r$ for any $b\in H$. Since $\Ad_{b}$ is an isomorphism it suffices to show that $\Ad_{b}(\mathfrak{a}_r)\subset\mathfrak{a}_r$. Let $u\in\mathfrak{a}_r$, then there exists $\eta\in\mathfrak{h}^{\circ}$ such that $q(\eta^{\widetilde{\#}})=q(u)$. Hence
	\begin{align*}
	q(\Ad_{b}(u))&=\overline{\Ad}_b(q(u))\\
	&=\overline{\Ad}_b(q(\eta^{\widetilde{\#}}))\\
	&=q(\Ad_{b}(\eta^{\widetilde{\#}}))\\
	&\stackrel{(\ref{Ad invariance 2})}{=}q((\Ad_{b}^{*}\eta)^{\widetilde{\#}}).
	\end{align*}
	Thus $\Ad_{b}(u)\in\mathfrak{a}_r$. 
\end{proof}

From the $G$-invariance of the Poisson bivector field $\pi$ it follows that $E{:=}\mathrm{Im}\pi_{\#}\subset TM$ is a homogeneous $G$-vector subbundle. Hence we get the following isomorphism of homogeneous $G$-vector bundle 
\begin{equation*}
G\times _H  E_{\bar{e}}  \stackrel{\cong}{\longrightarrow} E,(g,u)\mapsto T_{\bar{e}}(\lambda_g)(u).
\end{equation*}
Now, from Lemma $\ref{invariance}$ it follows that for any $b\in H$ we have 
\begin{equation*}
\Ad_{b}(\mathfrak{a}_r)=\mathfrak{a}_r,\, \Ad_{b}(\mathfrak{h})=\mathfrak{h}.
\end{equation*} Then we get a linear representation
$
\overline{\Ad}:H\rightarrow\mathrm{End}(\mathfrak{a}_r/\mathfrak{h})$.
Hence we get a a homogeneous $G$-vector bundle:
$
G\times_H (\mathfrak{a}_r/\mathfrak{h})\rightarrow M$.
\begin{theorem} The regular symplectic foliation $E$ is given by the homogenous $G$-vector bundle isomorphism
	\begin{equation*}
	G\times_H (\mathfrak{a}_r/\mathfrak{h})\stackrel{\cong}{\longrightarrow} E,\, (g,u+\mathfrak{h})\mapsto T_{\bar{e}}(\lambda_g)\circ T_ep(u).
	\end{equation*}
\end{theorem}
\begin{proof}
	Denote by $\psi:\mathfrak{a}_r/\mathfrak{h}\rightarrow E_{\bar{e}},\, u+\mathfrak{h}\mapsto T_ep(u)$. Obviously $\psi$ is a linear isomorphism of vector spaces. In particular we get the following commutative diagram 
	$$ \begin{tikzpicture}[ampersand replacement=\&]
	\matrix (m) [matrix of math nodes,row sep=3em,column sep=4em,minimum width=2em]
	{
		\mathfrak{a}_r/\mathfrak{h}  \& 	\mathfrak{a}_r/\mathfrak{h} \\
		E_{\bar{e}}\& E_{\bar{e}} \\};
	\path[-stealth]
	(m-1-1) edge node [left] {$\psi$} (m-2-1)
	edge node [above] {$\overline{\Ad}_a$} (m-1-2)
	edge node [below] {} (m-1-2)
	(m-2-1) edge node [above] {$T_{\bar{e}} (\lambda_a)$} (m-2-2)
	(m-1-2) edge node [right] {$\psi$} (m-2-2);
	\end{tikzpicture}$$
	Hence the two linear representations $\overline{\lambda}:H\rightarrow\mathrm{End}(E_{\bar{e}})$ and $\overline{\Ad}:H\rightarrow\mathrm{End}(\mathfrak{a}_r/\mathfrak{h})$ are equivalent. This means that the bundle map 
	\begin{equation*}
	G\times_H (\mathfrak{a}_r/\mathfrak{h})\longrightarrow E,\, (g,u+\mathfrak{h})\mapsto T_{\bar{e}}(\lambda_g)\circ T_ep(u).
	\end{equation*}
	is an isomorphism of homogeneous $G$-vector bundles over $M$.
\end{proof}

In the following we describe the leaf spaces.
\begin{proposition}\label{leaf spaces}$\ $
	\begin{enumerate}
		\item The symplectic leaf $\mathcal{F}^{\bar{g}}$ passing through $\bar{g}\in M$ is given by
		\begin{equation*}
		\mathcal{F}^{\bar{g}}=\{gaH,\, a\in A_r\}=c_g(A_r).gH,
		\end{equation*} 
		where $c_g(a)=gag^{-1}$.
		\item For any $\bar{g}\in M$ the symplectic leaf $\mathcal{F}^{\bar{g}}$ is a $c_{g}(A_r)$-homogeneous symplectic manifold, which  is isomorphic to the symplectic homogeneous space $A_r/H$.
		\item The leaf spaces $M/\mathcal{F}$ can be identified with $G/A_r$ through the map $gA_r\mapsto\mathcal{F}^{\bar{g}}$.
		\item If we assume that $(G,H)$ is a reductive pair with decomposition:
		\begin{equation*}
		\mathfrak{g}=\mathfrak{h}\oplus\mathfrak{m},\,\,\mathrm{Ad}(H)(\mathfrak{m})=\mathfrak{m}.
		\end{equation*}
		Then $(A_r,H)$ is a reductive pair with decomposition:
		\begin{equation*}
		\mathfrak{a}_r=\mathfrak{h}\oplus\mathrm{Im}(r_{\#}),\,\mathrm{Ad}(H)(
		\mathrm{Im}(r_{\#}))=\mathrm{Im}(r_{\#}).
		\end{equation*}
		\item If we assume that $(G,H)$ is a symmetric pair with canonical decomposition:
		\begin{equation*}
		\mathfrak{g}=\mathfrak{h}\oplus\mathfrak{m},\,\mathrm{Ad}(H)(\mathfrak{m})=\mathfrak{m},\,[\mathfrak{m},\mathfrak{m}]\subset\mathfrak{h}.
		\end{equation*}
		Then $(A_r,H)$ is a symmetric pair with canonical decomposition:
		\begin{equation*}
		\mathfrak{a}_r=\mathfrak{h}\oplus\mathrm{Im}(r_{\#}),\,\mathrm{Ad}(H)(\mathrm{Im}(r_{\#}))=\mathrm{Im}(r_{\#}),\,[\mathrm{Im}(r_{\#}),\mathrm{Im}(r_{\#})]\subset\mathfrak{h}.
		\end{equation*}
	\end{enumerate}
\end{proposition}
\begin{proof}
	\begin{enumerate}
		\item Let $b\in H$, 
		\begin{equation*}
		c_{gb}(A_r).gbH=c_g(c_b(A)).gH=c_g(A_r).gH.
		\end{equation*} Hence $c_g(A_r).gH$ is well defined. Now, let's compute $T_{\bar{e}}\mathcal{F}^{\bar{e}}$.
		\begin{align*}
		T_{\bar{e}}\mathcal{F}^{\bar{e}}&=T_e p(\mathfrak{a}_r)\\
		&=\Phi_{e}\circ q(\mathfrak{a}_r)\\
		&=\Phi_e\circ r_{\#}(\left(\mathfrak{g}/\mathfrak{h}\right)^{*})\\
		&=\Phi_e\circ r_{\#}\circ \Phi_{e}^{*}(T_{\bar{e}}^{*}M)\\
		&=\pi_{\#,\bar{e}}(T_{\bar{e}}^{*}M).
		\end{align*}
		Hence for any $a\in A_r$,
		\begin{align*}
		T_{\bar{a}}\mathcal{F}^{\bar{e}}&=T_{\circ}(\lambda_a)(T_{\bar{e}}\mathcal{F}^{\bar{e}})\\
		&=T_{\bar{e}}(\lambda_a)\circ \pi_{\#,\bar{e}}(T_{\bar{e}}^{*}M)\\
		&=T_{\bar{e}}(\lambda_a)\circ \pi_{\#,\bar{e}}\circ T_{\bar{e}}^{*}(\lambda_a)(T_{\bar{a}}^{*}M)\\
		&=\pi_{\#,\bar{a}}(T_{\bar{a}}^{*}M).
		\end{align*}
		This shows that the leaf passing through ${\bar{e}}$ is given by 
		\begin{equation*}
		\mathcal{F}^{\bar{e}}=\{aH,\, a\in A_r\}=A_r. H \, .
		\end{equation*} Hence the leaf passing through any $\bar{g}\in M$ is given by 
		\begin{equation*}
		\mathcal{F}^{\bar{g}}=\lambda_g(\mathcal{F}^{\bar{e}})=\{gaH,\, a\in A_r\}=c_{g}(A_r).H \, 
		\end{equation*}
		\item Let $g\in G$. Since $c_{g}(A_r)$ is a subgroup of $G$ it follows that $c_{g}(A_r)$ act on $\mathcal{F}^{\bar{g}}$ by Poisson transformation, which also act transitively on $\mathcal{F}^{\bar{g}}$. Hence $\mathcal{F}^{\bar{g}}$ is a $c_{g}(A_r)$-homogeneous symplectic manifold.
		\item Let $g'=ga$, where $a\in A_r$,
		\begin{equation*}
		\mathcal{F}^{\bar{g'}}=\{gaa'H,\, a'\in A_r\}=\{gaH,\, a\in A_r\}=\mathcal{F}^{\bar{g}}.
		\end{equation*}
		This means the map $gA_r\mapsto\mathcal{F}^{\bar{g}}$ is well defined.
		
		The other assumptions are obvious.
	\end{enumerate}
\end{proof}
As a corollary of Proposition \ref{leaf spaces} we get. 
\begin{corollary}
	The following assertions are equivalent:
	\begin{enumerate}
		\item $A_r/H$ is closed in $M$.
		\item $A_r$ is closed in $G$.
		\item The leaf space $M/A_r$ is a Hausdorff space. 
	\end{enumerate}
\end{corollary}
\begin{example}
	We have seen in Example \ref{Grassman manifold} that the bivector given by 
	\[ r=(e_1-e_4)\wedge (e_2+e_3)\]
	define a $\mathrm{SO}_{4}(\mathbb{R})$-invariant Poisson structures on the oriented Grassman manifold $\mathrm{Gr}^{+}_{2}(\mathbb{R}^n) \cong \mathrm{SO}_{4}(\mathbb{R})/(\mathrm{SO}_{2}(\mathbb{R})\times \mathrm{SO}_{2}(\mathbb{R}))$. A direct computation shows that 
	\[\mathfrak{a}_r=\left\{\begin{bmatrix}
	0 & x & z & t \\
	-x & 0 & t & -z \\
	-z & -t & 0 & y \\
	-t & z & -y & 0
	\end{bmatrix}\in\mathfrak{gl}_4(\R),\,\, x,y,z,t\in\mathbb{R}\right\}, \]
	and the Lie subgroup $A_r$ of $\mathrm{SO}_{4}(\mathbb{R})$ which integrate $\mathfrak{a}_r$ is given by
	\[A_r=\left\{\begin{bmatrix}
	a & b & c & d \\
	-b & a & d & -c \\
	c' & d' & a' & b' \\
	d' & -c' & -b' & a'
	\end{bmatrix}\in\mathrm{GL}_4(\R),\text{ such that }  
	\begin{array}{llll}
	a^2+b^2+c'^2+d'^2=1 \\
	a'^2+b'^2+c^2+d^2=1 \\
	ac-bd+a'c'-b'd'=0 \\
	ad+bc+c'b'+a'd'=0
	\end{array}
	\right\}. \]
\end{example}

\section{Invariant contravariant connections on $(G/H,\pi)$}\label{Section4} 
Let $(M,\pi)$ be a Poisson manifold. The concept of a contravariant connection, originally introduced as the contravariant derivative in \cite{VAISMAN} (p. 55), was later studied from a geometric perspective in \cite{FERNANDES}. For applications of this notion in the context of Poisson manifolds equipped with a compatible pseudo-Riemannian metric, we refer the reader to \cite{Boucetta1}.
\\ Nomizu's theorem \cite{Nomizu} on invariant covariant connections on reductive homogeneous spaces naturally leads us to pose the following question: {\it {Can we provide an algebraic description of invariant contravariant connections on reductive homogeneous spaces endowed with an invariant Poisson structure?}} Addressing this question is the primary objective of this section.\\
We recall that a contravariant connection on $(M,\pi)$ is an $\mathbb{R}$-bilinear map 
	$$D:\Omega^1(M)\times\Omega^1(M)\rightarrow\Omega^1(M),\, (\alpha,\beta)\mapsto D_{\alpha}\beta,$$
	satisfying, for any $f\in C^{\infty}(M)$,
	\begin{enumerate}
		\item  $D_{f\alpha}\beta=fD_{\alpha}\beta$,
		\item $D_{\alpha}f\beta=(\alpha^\#\cdot f)\beta+fD_{\alpha}\beta$.
	\end{enumerate}

\begin{remark}\label{RCOVCON}
	To any covariant connection $\nabla$ on $M$, one can define a contravariant connection on $(M,\pi)$ by setting $
	\nabla^{\#}_{\alpha}\beta:=\nabla_{\alpha^\#}\beta.
$.
Such contravariant connections form a subclass of $\mathcal{F}$-connections, characterized by the following condition:
  \[
\alpha^\#=0\implies \nabla^{\#}_{\alpha}=0.\]
\end{remark}
 Analogously to the covariant case's, the torsion $T$ and the curvature $R$ of a linear contravariant connection $D$ are defined by
 \begin{enumerate}
 	\item[$\bullet$] $T(\alpha,\beta)=D_{\alpha}\beta-D_{\beta}\alpha-[\alpha,\beta]_{\pi}$
 	\item[$\bullet$] $R(\alpha,\beta)\gamma=D_{\alpha}D_{\beta}\gamma-D_{\beta}D_{\alpha}\gamma-D_{[\alpha,\beta]_{\pi}}\gamma$.
 \end{enumerate}
\begin{definition}
	Let $G/H$ be a $G$-homogeneous space endowed with a Poisson tensor $\pi$. A contravariant connection $D$ on $G/H$ is $G$-invariant if, for any $g\in G$, and $\alpha,\beta\in\Omega^1(G/H)$, we have
	\begin{equation*}
		g\cdot D_{\alpha}\beta=D_{g\cdot\alpha}g\cdot\beta,
	\end{equation*}
where  $g\cdot\alpha:=\lambda_{g^{-1}}^{*}\alpha$.
\end{definition}

Now, let $(G,H)$ be a reductive pair with fixed decomposition: $\mathfrak{g}=\mathfrak{h}\oplus\mathfrak{m},\, \mathrm{Ad}_H(\mathfrak{m})=\mathfrak{m}$.\\
In what follows,
\begin{itemize}
	\item $\pi$ is a $G$-invariant Poisson structures on $G/H$ and  $r\in(\wedge^{2}\mathfrak{m})^{H}$ its associated $H$-invariant bivector.
	\item For any one form $\beta\in\Omega^1(G/H)$, $F^{\beta}:G\rightarrow\mathfrak{m}^{*}$ is the $H$-equivariant function defined by
	\begin{equation*}
		F^\beta(g)=(g^{-1}\cdot \beta)_{\bar{e}}=(T_{\bar{e}}\lambda_{g})^{*}\beta_{\bar{g}},
	\end{equation*} 
	and for any $u\in\mathfrak{g}$, 
	\begin{equation*}
		u\cdot F^{\beta}=(dF^{\beta})_{e}(-u).
	\end{equation*}
	\item We will identify $T_{\bar{e}}(G/H)$ with $\mathfrak{m}$ thought the map $u^{*}_{\bar{e}}\mapsto u$, consequently  $T_{\bar{e}}^{*}(G/H)$  will be identified with $\mathfrak{m}^*$.
\end{itemize}
Since $G/H$ is a reductive homogeneous $G$-space, then according to Nomizu theorem (\cite{Nomizu}), there is a one-to-one correspondence between the set of $G$-invariant covariant connections on $G/H$ and the set of $\mathrm{Ad}(H)$-invariant bilinear maps $\mathfrak{\psi}:\mathfrak{m}\times\mathfrak{m} \rightarrow\mathfrak{m}$. In the context of contravariant connections we will prove the following result.
\begin{theorem}\label{ILCC}
 There is a one-to-one correspondence between $G$-invariant contravariant connections on $(G/H,\pi)$ and  $\mathrm{Ad}(H)$-invariant bilinear maps $\mathfrak{b}:\mathfrak{m}^*\times\mathfrak{m}^* \rightarrow\mathfrak{m}^*$, that is, $\mathrm{Ad}_{a}^{*}\mathfrak{b}(\eta,\xi)=\mathfrak{b}(\mathrm{Ad}_{a}^{*}\eta,\mathrm{Ad}_{a}^{*}\xi)$ for any $\eta,\xi\in\mathfrak{m}^*$ and $a\in H$. The $G$-invariant contravariant connection $D$ corresponding to $\mathfrak{b}$ is given by
 \begin{equation*}
 	\displaystyle \left(D_{\alpha}\beta\right)(\bar{e})=\mathfrak{b}(\alpha_{\bar{e}},\beta_{\bar{e}})+\alpha_{\bar{e}}^{\#}\cdot F^{\beta}.
 \end{equation*}
\end{theorem}
To prove theorem \ref{ILCC} we need the following lemmas. 
\begin{lemma}\label{L1ICC}
	For any $u\in\mathfrak{m}$, $f\in C^{\infty}(G/H)$ and $\beta\in\Omega^1(G/H)$ we have
	\begin{equation*}
		u\cdot F^{f\beta}=u^{*}_{\bar{e}}\cdot f F^{\beta}(e)+f(\bar{e})u\cdot F^{\beta}.
	\end{equation*}
\end{lemma}
\begin{proof} Let $u\in\mathfrak{m}$, $f\in C^{\infty}(G/H)$ and $\beta\in\Omega^1(G/H)$. Then we have
	\begin{align*}
			u\cdot F^{f\beta}&=\left.\frac{d}{d t}\right|_{t=0} F^{f\beta}(\exp (- t u)) \\
			&=\left.\frac{d}{d t}\right|_{t=0} (\exp(tu)\cdot (f\beta))_{\bar{e}} \\
			&= \left.\frac{d}{d t}\right|_{t=0} f(\exp(-tu)H)(\exp(tu)\cdot \beta)_{\bar{e}}\\
			&=u^{*}_{\bar{e}}\cdot f\beta_{\bar{e}}+f(\bar{e})u\cdot F^{\beta}.
	\end{align*}
\end{proof}
\begin{lemma}\label{L2ICC}
	For any $u\in\mathfrak{m}$, $\beta\in\Omega^1(G/H)$  and $a\in H$ we have
	\begin{equation*}
		\mathrm{Ad}_{a} (u)\cdot F^{a\cdot \beta}=\mathrm{Ad}_{a}^{*}\left(u\cdot F^{\beta}\right).
	\end{equation*}
\end{lemma}
\begin{proof}
Let $u\in\mathfrak{m}$, $\beta\in\Omega^1(G/H)$  and $a\in H$. Then we have
	\begin{align*}
		\mathrm{Ad}_{a} (u)\cdot F^{a\cdot \beta}&=\left.\frac{d}{d t}\right|_{t=0} F^{a\cdot\beta}(\exp (-t \mathrm{Ad}_{a} (u))) \\
		&=\left.\frac{d}{d t}\right|_{t=0} (\exp(t\mathrm{Ad}_{a} (u))\cdot (a\cdot\beta))_{\bar{e}.} \\
		&= \left.\frac{d}{d t}\right|_{t=0} (a\cdot\exp(-t u)\cdot a^{-1}\cdot (a\cdot\beta))_{\bar{e}}\\
		&=\left.\frac{d}{d t}\right|_{t=0} (a\cdot(\exp(t u)\cdot\beta))_{\bar{e}}\\
		&=\mathrm{Ad}_{a}^{*}\left(\left.\frac{d}{d t}\right|_{t=0} F^{\beta}(\exp (-t u))\right)\\
		&=\mathrm{Ad}_{a}^{*}\left(u\cdot F^{\beta}\right)
	\end{align*}
\end{proof}
\begin{proof}[Proof of theorem \ref{ILCC}]
	 Let $D$ be a $G$-invariant contravariant connection. We define a bilinear map  $\mathfrak{b}:\mathfrak{m}^*\times\mathfrak{m}^* \rightarrow\mathfrak{m}^*$ by setting
 \begin{equation*}
	\displaystyle \mathfrak{b}(\eta,\xi):=\left(D_{\alpha}\beta\right)(\bar{e})-\alpha_{\bar{e}}^{\#}\cdot F^{\beta},
\end{equation*}
where $\alpha$ and $\beta$ are any one forms on $G/H$ satisfying $\alpha_{\bar{e}}=\eta$ and   $\beta_{\bar{e}}=\xi$. As a first step lets show that  $\mathfrak{b}$ is well defined.
\begin{enumerate}
	\item Suppose that $\alpha_{\bar{e}}=0$. Then we have 
	\begin{equation*}
		\left(D_{\alpha}\beta\right)(\bar{e})=D_{\alpha_{\bar{e}}}\beta=0 \text{ and }\alpha_{\bar{e}}^{\#}\cdot F^{\beta}=(dF^{\beta})_{e}(\alpha_{\bar{e}}^{\#})=0.
	\end{equation*}
	\item Suppose that $\beta_{\bar{e}}=0$. One can see easily that there exists an open neighborhood $U$ of $\bar{e}\in G/H$ and a smooth functions $(f_i)_{1 \leq i\leq m}\in C^{\infty}(U)$ and a one forms $(\beta_i)_{1 \leq i\leq m}\in \Omega^{1}(U)$ such that 
	$$\displaystyle \beta=\sum_{i=1}^{m}f_i\beta_i,\text{ and } f_i(\bar{e})=0, \text{ for }i=1,\ldots,m.$$ 
	Hence from Lemma \ref{L1ICC} it follows that
	\begin{align*}
		\displaystyle\left(D_{\alpha}\beta\right)(\bar{e})-\alpha_{\bar{e}}^{\#}\cdot F^{\beta}&=\sum_{i=1}^{m}(\alpha_{\bar{e}}^{\#}\cdot f_i)F^{\beta_{i}}(e)+f_i.(\bar{e})\left(D_{\alpha}\beta_i\right)(\bar{e})\\
		&\quad -(\alpha_{\bar{e}}^{\#}\cdot f_i)F^{\beta_{i}}(e)-f_i(\bar{e})\alpha_{\bar{e}}^{\#}\cdot F^{\beta_i}\\
		&=\sum_{i=1}^{m}f_i(\bar{e})\left(D_{\alpha}\beta_i\right)(\bar{e})-f_i(\bar{e})\alpha_{\bar{e}}^{\#}\cdot F^{\beta_i}\\
		&=0.
	\end{align*}	 
\end{enumerate}
From (\ref{Carac2}) and Lemma \ref{L2ICC} it follows that $\mathfrak{b}$ is $\mathrm{Ad}(H)$-invariant. Indeed, let  $a\in H$,
\begin{align*}
\displaystyle	\mathfrak{b}(\mathrm{Ad}^{*}_{a}\eta,\mathrm{Ad}^{*}_{a}\xi)&=\left(D_{a\cdot\alpha}a\cdot\beta\right)(\bar{e})-(\Ad_{a}^{*}\alpha_{\bar{e}})^{\#})\cdot F^{a\cdot\beta}\\
&=\left(a\cdot D_{\alpha}\beta\right)(\bar{e})-\Ad_{a}(\alpha_{\bar{e}}^{\#})\cdot F^{a\cdot\beta}\\
&=\mathrm{Ad}^{*}_{a}\left(\displaystyle\left(D_{\alpha}\beta\right)(\bar{e})-\alpha_{\bar{e}}^{\#}\cdot F^{\beta}\right)\\
&=\mathrm{Ad}_{a}^{*}\mathfrak{b}(\eta,\xi).
\end{align*}
Conversely, let $\mathfrak{b}:\mathfrak{m}^*\times\mathfrak{m}^* \rightarrow\mathfrak{m}^*$ be an $\mathrm{Ad}(H)$-invariant bilinear map. then the $G$-invariant contravariant connection $D$ giving by 
 \begin{equation*}
	\displaystyle \left(D_{\alpha}\beta\right)(\bar{e}):=\mathfrak{b}(\alpha_{\bar{e}},\beta_{\bar{e}})+\alpha_{\bar{e}}^{\#}\cdot F^{\beta},
\end{equation*}
is well defined.
\end{proof}
\begin{corollary} If $\pi$ is a left invariant Poisson tensor on $G$, then left invariant  contravariant connections on $(G,\pi)$ are in bijective correspondence with bilinear maps $\mathfrak{b}:\mathfrak{g}^*\times\mathfrak{g}^* \rightarrow\mathfrak{g}^*$. The left invariant contravariant connection $D$ corresponding to $\mathfrak{b}$ is given by
	\begin{equation*}
		\displaystyle \left(D_{\eta^l}\xi^l\right)(e)=\mathfrak{b}(\eta,\xi),
	\end{equation*}
where $\eta,\xi\in\mathfrak{g}^*$.
\end{corollary}
\begin{proof}
Let $\eta,\xi\in\mathfrak{g}^*$, then we have
\begin{equation*}
	F^{\xi^l}(g)=(g^{-1}\cdot \xi^l)_{e}=\xi.
\end{equation*}
Hence we get 
\begin{eqnarray*}
	(\eta^l)_{e}^{\#}\cdot F^{\xi^l}=0.
\end{eqnarray*}  
This proves the desired correspondence.
\end{proof}

Let $\nabla$ be a $G$-invariant covariant connection on $G/H$. It is clear that the associated contravariant connection $\nabla^\#$ is $G$-invariant. Then, by Theorem \ref{ILCC}, we can associate to it a bilinear product $\mathfrak{b}:\mathfrak{m}^*\times\mathfrak{m}^*\rightarrow\mathfrak{m}^*$.
On the other hand, we know by Nomizu's Theorem, that the covariant connection  $\nabla$ is characterized by an $\Ad(H)$-invariant bilinear map $\mathfrak{b}^{\nabla}:\mathfrak{m}\times\mathfrak{m}\rightarrow\mathfrak{m}$.
\begin{proposition} The two products $\mathfrak{b}$ and $\mathfrak{b}^{\nabla}$ are connected by the following formula
	\[ \mathfrak{b}(\eta,\xi)=(\mathfrak{b}^{\nabla}_{\eta^{\#}})^*\xi, \qquad \forall \eta,\xi\in \mathfrak{m}^*.\]
\end{proposition}
\begin{proof}
	From formula $(4.9)$ in \cite[p. 8]{abb} we get that for any $\alpha,\beta\in\Omega^1(G/H)$,
	\begin{equation*}
		(\nabla_{\alpha^\#}\beta)(\bar{e})=(\mathfrak{b}^{\nabla}_{\alpha_{\bar{e}}^{\#}})^{*}\beta_{\bar{e}}+\alpha^{\#}_{\bar{e}}.F^{\beta}.
	\end{equation*}
	Hence
	$$
	\mathfrak{b}(\alpha_{\bar{e}},\beta_{\bar{e}})=(\nabla^\#_{\alpha}\beta)(\bar{e})-\alpha^{\#}_{\bar{e}}\cdot F^{\beta}=(\mathfrak{b}^{\nabla}_{\alpha_{\bar{e}}^{\#}})^{*}\beta_{\bar{e}}.
	$$
\end{proof}

In general, an $\mathcal{F}$-connection is not induced by a covariant connection \cite{FERNANDES}. However, in our case, we will see that such connections are equivalent. 
\begin{theorem} Let $D$ be a $G$-invariant contravariant connection on $(G/H, \pi)$, and let $$\mathfrak{b} : \mathfrak{m}^* \times \mathfrak{m}^* \to \mathfrak{m}^*, \qquad (\eta, \xi) \mapsto \mathfrak{b}_{\eta}(\xi) = \mathfrak{b}(\eta, \xi),$$ be its associated bilinear map. The following assertions are equivalents:
	\begin{enumerate}
		\item $D$ is an $\mathcal{F}$-connection.
		\item For any $\eta\in\mathfrak{m}^*$, we have:
		$\eta^\#=0$ implies $\mathfrak{b}_{\eta}=0$.
		\item $D$ is induced by a $G$-invariant covariant connection on $G/H$.
	\end{enumerate} 
\end{theorem}
\begin{proof} The only implication that needs to be shown is that 2. implies 3. Indeed, suppose that $2$ is satisfied. Let $V$ be a complementary subspace of $\mathrm{Im}(r_{\#})$ in $\mathfrak{m}$. Consider the bilinear map $\mu: \mathfrak{m} \times \mathfrak{m} \to \mathfrak{m}$ defined by: For $u=v+w$ with $w \in \mathrm{Im}(r_{\#}) $ and $v\in V$
	\[
	(\mu_{u})^*= \mathfrak{b}_\eta,
	\]
	where $\eta \in \mathfrak{m}^*$ is any element satisfying $\eta^{\#} = w$.\\
	Clearly, $\mu$ is well-defined and $\mathrm{Ad}(H)$-invariant. Hence, it induces a $G$-invariant covariant connection $\nabla$ on $G/H$. Then, we have  $\mathfrak{b}_\eta=(\mu_{\eta^{\#} })^*$ which gives that $D=\nabla^{\#}$. Indeed, Let $\alpha,\beta$ be two differential $1$-forms on $G/H$,
	$$
	(D_\alpha\beta)_{\bar{e}}= \mathfrak{b}(\alpha_{\bar{e}},\beta_{\bar{e}})+\alpha_{\bar{e}}^{\#}\cdot F^{\beta}=
	(\mu_{\alpha_{\bar{e}}^{\#}})^*(\beta_{\bar{e}})+\alpha_{\bar{e}}^{\#}\cdot F^{\beta}=(\nabla^{\#}_\alpha\beta)_{\bar{e}}.
	$$
	
\end{proof}
	Similarly to the covariant case we have the followig characteriztion of the torsion and the curvature of $G$-invariant contravariant connections.
\begin{theorem}\label{TorCur}
	Let $D$ be a $G$-invariant  $G$-invariant contravariant connections on $(G/H,\pi)$. Then the torsion and the curvature of $D$ are given by
	\begin{equation}\label{Torsion}
		T(\eta,\xi)=\mathfrak{b}(\eta,\xi)-\mathfrak{b}(\xi,\eta)-[\eta,\xi]_r,
	\end{equation}
\begin{equation}\label{Curvature}
	R(\eta,\xi)=[\mathfrak{b}_{\eta},\mathfrak{b}_{\xi}]-\mathfrak{b}_{[\eta,\xi]_r},
\end{equation}
where $\eta,\xi\in\mathfrak{m}$ and $[\,,\,]_r$ is given by equation (\ref{reductive equation}).
\end{theorem} 
To prove such characterizations we need the following lemmas.
\begin{lemma}\label{bracketlemma}
For any $\alpha,\beta\in \Omega^1(G/H)$,
\begin{equation}
	[\alpha,\beta]_\pi(\bar{e})=\eta^{\#}\cdot F^{\beta}-\xi^{\#}\cdot F^{\alpha}+[\eta,\xi]_{r},
\end{equation}
where $\eta=\alpha_{\bar{e}}$ and $\xi=\beta_{\bar{e}}$.
\end{lemma}
\begin{proof}
Let $u\in\mathfrak{g}$ and $\alpha\in\Omega^1(G/H)$. Then from the $G$-invariance of $\pi$ we get
\begin{equation*}
	(\mathcal{L}_{u^*}\alpha)^\#=[u^*,\alpha^\#].
\end{equation*} 
	Hence for any $\beta\in\Omega^1(G/H)$,
	\begin{align*}
		\langle [\alpha,\beta]_\pi,u^*\rangle&=\langle \mathcal{L}_{\alpha^{\#}}\beta,u^*\rangle-\langle \mathcal{L}_{\beta^{\#}}\alpha,u^*\rangle-u^*\cdot \pi(\alpha,\beta)\\
		&=\alpha^\#\cdot \langle \beta,u^*\rangle- \langle \beta, [\alpha^\#,u^*]\rangle\\
		&\quad -\beta^\#\cdot \langle \alpha,u^*\rangle+ \langle \alpha, [\beta^\#,u^*]\rangle-u^*\cdot \pi(\alpha,\beta)\\
		&=\alpha^\#\cdot \langle \beta,u^*\rangle-\beta^\#\cdot \langle \alpha,u^*\rangle\\
			&\quad -u^*\cdot \pi(\alpha,\beta)+\pi(\mathcal{L}_{u^*}\alpha,\beta)+\pi(\alpha,\mathcal{L}_{u^*}\beta)\\
			&=\alpha^\#\cdot \langle \beta,u^*\rangle-\beta^\#\cdot \langle \alpha,u^*\rangle.
	\end{align*}
Let $\eta=\alpha_{\bar{e}}$ and $\xi=\beta_{\bar{e}}$ and  $v=-\eta^{\#},w=-\xi^{\#}\in\mathfrak{g}$, then we have
\[
v^{*}_{\bar{e}}=\alpha^{\#}_{\bar{e}},\, w^{*}_{\bar{e}}=\beta^{\#}_{\bar{e}}.
\]  
Hence
\begin{align*}
	\langle [\alpha,\beta]_\pi(\bar{e}),u^{*}_{\bar{e}}\rangle
	&=v^{*}_{\bar{e}}\cdot \langle \beta,u^{*}\rangle-w^{*}_{\bar{e}}\cdot \langle \alpha,u^{*}\rangle\\
	&= \langle (\mathcal{L}_{v^*}\beta)(\bar{e}) ,u^{*}_{\bar{e}}\rangle+\langle \beta_{\bar{e}},[v^*,u^*]_{\bar{e}}\rangle \\
	&\quad -\langle (\mathcal{L}_{w^*}\alpha)(\bar{e}) ,u^{*}_{\bar{e}}\rangle-\langle \alpha_{\bar{e}},[w^*,u^*]_{\bar{e}}\rangle \\
	&=\langle \eta^{\#}\cdot F^\beta ,u\rangle-\langle \ad_{v}^{*}\widetilde{\xi},u\rangle\\
	&\quad -\langle \xi^{\#}\cdot F^\alpha ,u\rangle+\langle \ad_{w}^{*}\widetilde{\eta},u\rangle\\
	&= \langle \eta^{\#}\cdot F^{\beta}-\xi^{\#}\cdot F^{\alpha},u\rangle+\langle [\eta,\xi]_{r},u\rangle.
\end{align*}
\end{proof}
\begin{proof}[Proof of Theorem \ref{TorCur}]
	Let $\alpha,\beta\in\Omega^1(G/H)$.
	From Lemma \ref{bracketlemma} it follows that
	\begin{align*}
		T(\alpha,\beta)(\bar{e})&=(D_{\alpha}\beta)(\bar{e})-\alpha^{\#}_{\bar{e}}\cdot F^{\beta}-(D_{\beta}\alpha)(\bar{e})+\beta^{\#}_{\bar{e}}\cdot F^{\alpha}-[\alpha_{\bar{e}},\beta_{\bar{e}}]_{r}\\
		&=\mathfrak{b}(\alpha_{\bar{e}},\beta_{\bar{e}})-\mathfrak{b}(\beta_{\bar{e}},\alpha_{\bar{e}})-[\alpha_{\bar{e}},\beta_{\bar{e}}]_{r}.
	\end{align*}
Hence we get (\ref{Torsion}). 

Now lets prove (\ref{Curvature}), We consider $u,v\in\mathfrak{g}$ such that
$u^{*}_{\bar{e}}=\alpha^{\#}_{\bar{e}}$ and $v^{*}_{\bar{e}}=\beta^{\#}_{\bar{e}}$. Hence for any $\gamma\in\Omega^1(G/H)$ we have
	\begin{align*}
	R(\alpha,\beta)(\gamma)(\bar{e})&=(D_{\alpha}D_{\beta}\gamma)(\bar{e})-(D_{\beta}D_{\alpha}\gamma)(\bar{e})-(D_{[\alpha,\beta]_\pi}\gamma)(\bar{e})\\
	&=\mathfrak{b}(\alpha_{\bar{e}},(D_{\beta}\gamma)(\bar{e}))+\alpha_{\bar{e}}^{\#}\cdot F^{D_{\beta}\gamma}-\mathfrak{b}(\beta_{\bar{e}},(D_{\alpha}\gamma)(\bar{e}))-\beta_{\bar{e}}^{\#}\cdot F^{D_{\beta}\gamma}\\
	&\quad -\mathfrak{b}([\alpha,\beta]_{\pi}(\bar{e}),\gamma_{\bar{e}})-[\alpha,\beta]_{\pi}^{\#}(\bar{e})\cdot F^{\gamma}\\
	&=\mathfrak{b}(\alpha_{\bar{e}},\mathfrak{b}(\beta_{\bar{e}},\gamma_{\bar{e}}))+\mathfrak{b}(\alpha_{\bar{e}},(\mathcal{L}_{v^*}\gamma)(\bar{e}))+(\mathcal{L}_{u^*}D_{\beta}\gamma)(\bar{e})\\
	&\quad -\mathfrak{b}(\beta_{\bar{e}},\mathfrak{b}(\alpha_{\bar{e}},\gamma_{\bar{e}}))-\mathfrak{b}(\beta_{\bar{e}},(\mathcal{L}_{u^*}\gamma)(\bar{e}))-(\mathcal{L}_{v^*}D_{\alpha}\gamma)(\bar{e})\\
	&\quad -\mathfrak{b}([\alpha_{\bar{e}},\beta_{\bar{e}}]_r,\gamma_{\bar{e}})-\mathfrak{b}((\mathcal{L}_{u^*}\beta)(\bar{e}),\gamma_{\bar{e}})+\mathfrak{b}((\mathcal{L}_{v^*}\alpha)(\bar{e}),\gamma_{\bar{e}})-[\alpha^{\#},\beta^{\#}]_{\bar{e}}\cdot F^{\gamma}
\end{align*}
Since we have
\begin{align*}
	\mathfrak{b}(\alpha_{\bar{e}},(\mathcal{L}_{v^*}\gamma)(\bar{e}))&=\left(D_{\alpha}\mathcal{L}_{v^*}\gamma+\mathcal{L}_{u^*}\mathcal{L}_{v^*}\gamma\right)(\bar{e}),\\
	\mathfrak{b}(\beta_{\bar{e}},(\mathcal{L}_{u^*}\gamma)(\bar{e}))&=\left(D_{\beta}\mathcal{L}_{u^*}\gamma+\mathcal{L}_{v^*}\mathcal{L}_{u^*}\gamma\right)(\bar{e}),\\
	\mathfrak{b}((\mathcal{L}_{u^*}\beta)(\bar{e}),\gamma_{\bar{e}})&=\left(D_{\mathcal{L}_{u^*}\beta}\gamma\right)(\bar{e})-[u^*,\beta^{\#}]_{\bar{e}}\cdot F^{\gamma},\\
	\mathfrak{b}((\mathcal{L}_{v^*}\alpha)(\bar{e}),\gamma_{\bar{e}})&=\left(D_{\mathcal{L}_{v^*}\alpha}\gamma\right)(\bar{e})-[v^*,\alpha^{\#}]_{\bar{e}}\cdot F^{\gamma}.
\end{align*} 
From this last equations and the $G$-invariance of $D$ we get that

 \begin{align*}
 	R(\alpha,\beta)(\gamma)(\bar{e})
 	&=\mathfrak{b}(\alpha_{\bar{e}},\mathfrak{b}(\beta_{\bar{e}},\gamma_{\bar{e}}))-\mathfrak{b}(\beta_{\bar{e}},\mathfrak{b}(\alpha_{\bar{e}},\gamma_{\bar{e}}))-\mathfrak{b}([\alpha_{\bar{e}},\beta_{\bar{e}}]_r,\gamma_{\bar{e}})\\
 	&\quad+[u^*-\alpha^{\#},\beta^{\#}-v^*]_{\bar{e}}\cdot F^{\gamma}\\
 	&=\mathfrak{b}(\alpha_{\bar{e}},\mathfrak{b}(\beta_{\bar{e}},\gamma_{\bar{e}}))-\mathfrak{b}(\beta_{\bar{e}},\mathfrak{b}(\alpha_{\bar{e}},\gamma_{\bar{e}}))-\mathfrak{b}([\alpha_{\bar{e}},\beta_{\bar{e}}]_r,\gamma_{\bar{e}})
 \end{align*}
\end{proof}

 Inspiring from covariant case, we will give a distinguished classes of invariant contravariant connections. For this we need the following elements
 \begin{lemma}\label{Canonical connections} 
 	For any $\eta,\xi\in\mathfrak{m}^*$ and $a\in H$, we have
 	\begin{enumerate}
 		\item $(\Ad_{a}^{*}\eta)\circ l_{(\Ad_{a}^{*}\xi)^{\#}}=\Ad_{a}^{*}(\eta\circ l_{\xi^{\#}})$,  where $l_{\xi^{\#}}:\mathfrak{m}\rightarrow\mathfrak{m},\, u\mapsto [\xi^{\#},u]_{\mathfrak{m}}$.
 		\item $\eta\circ l_{\xi^{\#}}-\xi\circ l_{\eta^{\#}}=[\xi,\eta]_r$,
 		\item $[\Ad_{a}^{*}\eta,\Ad_{a}^{*}\xi]_r=\Ad_{a}^{*}[\eta,\xi]_r$.
 	\end{enumerate}
 \end{lemma}
\begin{proof}Let $\eta,\xi\in\mathfrak{m}^*$, and $u\in\mathfrak{m}$
	\begin{enumerate}
		\item Let $a\in H$,
		\begin{align*}
			\langle(\Ad_{a}^{*}\eta)\circ l_{(\Ad_{a}^{*}\xi)^{\#}},u\rangle &= 	\langle\Ad_{a}^{*}\eta,[ (\Ad_{a}^{*}\xi)^{\#},u]_{\mathfrak{m}}\rangle\\
			&=\langle\eta,[\xi^{\#},\Ad_{a^{-1}}(u)]_{\mathfrak{m}}\rangle\\
			&=\langle\Ad_{a}^{*}(\eta\circ l_{\xi^{\#}}),u\rangle.
		\end{align*}
	\item \begin{align*}
		\langle \eta\circ l_{\xi^{\#}}-\xi\circ l_{\eta^{\#}},u\rangle &=\langle\eta,[ \xi^{\#},u]_{\mathfrak{m}}\rangle-\langle\xi,[ \eta^{\#},u]_{\mathfrak{m}}\rangle\\
		&=-\langle\ad_{\xi^{\#}}^{*}\widetilde{\eta},u\rangle+\langle\ad_{\eta^{\#}}^{*}\widetilde{\xi},u\rangle\\
		&=\langle [\eta,\xi]_r,u\rangle.
	\end{align*}
\item It follows from $1$ and $2$.
	\end{enumerate}
\end{proof}
From Lemma \ref{Canonical connections} we get the following proposition.

\begin{proposition}$\ $
\begin{enumerate}
	\item The bilinear map
	\[ \mathfrak{b}(\eta,\xi)=-\xi\circ l_{\eta^{\#}},\]
	where $\eta,\xi\in\mathfrak{m}^*$, define a torsionless $G$-invariant contravariant connection. Moreover, $\mathfrak{b}$ induce a compatible left symmetric product on the Lie algebra $((\mathfrak{m}^*)^H,[\cdot,\cdot]_r)$.
	\item The bilinear map
\[ \mathfrak{b}(\eta,\xi)=\frac{1}{2}[\eta,\xi]_{r},\]
	where $\eta,\xi\in\mathfrak{m}^*$, define a torsionless $G$-invariant contravariant connection which will be called the naturel contravariant connection.
	\item The Canonical contravariant  connection, is given by \[ \mathfrak{b}(\eta,\xi)=0,\]
	where $\eta,\xi\in\mathfrak{m}^*$.
\end{enumerate}
\end{proposition}
\subsection*{Fedosov contravariant connection}
Recall that a contravariant connection $D$ on a Poisson manifold $(M,\pi)$ is Poisson if and only if, for all one forms $\alpha,\beta,\gamma\in\Omega^1(M)$,
\begin{equation}
	D\pi(\alpha,\beta,\gamma)=\alpha^{\#}.\pi(\beta,\gamma)-\pi(D_{\alpha}\beta,\gamma)-\pi(\beta,D_{\alpha}\gamma)=0.
\end{equation}
Now we consider a $G$-invariant contravariant connection $D$ on $(G/H,\pi)$ and denote by  $\mathfrak{b}:\mathfrak{m}^*\times\mathfrak{m}^*\rightarrow\mathfrak{m}^*$ its associated bilinear map. 
\begin{theorem}
 $D$ is Poisson if and only if, for any $\eta,\xi,\varepsilon\in\mathfrak{m}^*$,
 \begin{equation}\label{Poisson compatibility}
 	r(\mathfrak{b}(\eta,\xi),\varepsilon)+	r(\xi,\mathfrak{b}(\eta,\varepsilon))=0.
 \end{equation}
\end{theorem}
\begin{proof} From the $G$-invariance of $D$ and $\pi$ it follows that the tensor $D\pi$ is $G$-invariant. So lets compute the value of $D\pi$ at $\bar{e}$. Let $\alpha,\beta,\gamma\in\Omega^1(G/H)$ and let $u\in\mathfrak{g}$ such that $u^{*}_{\bar{e}}=\alpha^{\#}$, then we have
	\begin{align*}
		D\pi(\alpha,\beta,\gamma)(\bar{e})&=-\pi(D_{\alpha}\beta-\mathcal{L}_{u^*}\beta,\gamma)(\bar{e})-\pi(\beta,D_{\alpha}\gamma-\mathcal{L}_{u^*}\gamma)(\bar{e})\\
		&=-\pi((D_{\alpha}\beta)_{\bar{e}}-\alpha^{\#}_{\bar{e}}\cdot F^{\beta},\gamma_{\bar{e}})-\pi(\beta_{\bar{e}},(D_{\alpha}\gamma)_{\bar{e}}-\alpha^{\#}_{\bar{e}}\cdot F^{\gamma})\\
		&=-r(\mathfrak{b}(\alpha_{\bar{e}},\beta_{\bar{e}}),\gamma_{\bar{e}})-r(\beta_{\bar{e}},\mathfrak{b}(\alpha_{\bar{e}},\gamma_{\bar{e}})).
	\end{align*}
\end{proof}

\begin{example}[Riemannian contravariant connection]
A \emph{Riemannian Poisson manifold} is a smooth manifold $(M, \pi, \langle\cdot,\cdot\rangle)$ equipped with a Poisson structure $\pi \in \Gamma(\wedge^2 TM)$ and a Riemannian metric $\langle\cdot,\cdot\rangle$ such that the pair $(\pi, \langle\cdot,\cdot\rangle)$ are compatible in the sense of \cite{Boucetta}. Specifically, this means $D\pi = 0$, where $D$ is the Levi-Civita contravariant connection given by
\begin{align*}
	2\langle D_{\alpha}\beta,\gamma \rangle^* &= \alpha^{\#}\cdot\langle \beta,\gamma \rangle^* 
	+ \beta^{\#}\cdot\langle \alpha,\gamma\rangle^* 
	- \gamma^{\#}\cdot\langle \alpha,\beta\rangle^*  \\
	&\quad + \langle [\alpha^{\#}, \beta^{\#}]_\pi, \gamma\rangle^* 
	+ \langle [\gamma^{\#}, \beta^{\#}]_\pi, \alpha\rangle^* 
	+ \langle [\gamma^{\#}, \alpha^{\#}]_\pi, \beta\rangle^*,
\end{align*}
for all \(\alpha, \beta, \gamma \in \Omega^1(M)\), where $\langle\cdot,\cdot\rangle^*$ is the co-metric associated with the Riemannian metric $\langle\cdot,\cdot\rangle$.

In the case where $M = G/H$ is a reductive homogeneous space, and $\pi$ and $\langle\cdot,\cdot\rangle$ are both $G$-invariant, the Levi-Civita contravariant connection $D$ will also be $G$-invariant. Consequently, $(\pi, \langle\cdot,\cdot\rangle)$ are compatible if and only if the associated bilinear map $\mathfrak{b}$ satisfies Equation (\ref{Poisson compatibility}). In particular, in the case of a Lie group, we recover Proposition 2.1 in \cite{Alioune}.
\end{example}
Recall that a Fedosov connection on a symplectic manifold $(M,\omega)$ is a torsionless connection $\nabla$ such that $\nabla\omega=0$. Analogously to the covariant case we have.
\begin{definition}
	A  Fedosov contravariant connection on a Poisson manifold $(M,\pi)$ is a torsionless Poisson connection.
\end{definition}
As natural example of Fedosov contravariant connection on $(G/H,\pi)$ we have.
\begin{proposition}
The bilinear map
	\[ \mathfrak{b}(\eta,\xi)=\frac{1}{3}\left([\eta,\xi]_r-\xi\circ l_{\eta^{\#}}\right),\]
	for $\eta,\xi\in\mathfrak{m}^*$, defines a $G$-invariant  Fedosov contravariant connection on $(G/H,\pi)$.
\end{proposition} 
\begin{proof}
	Let $\eta,\xi,\varepsilon\in\mathfrak{m}^*$,
	\begin{align*}
		r(\mathfrak{b}(\eta,\xi),\varepsilon)&=\frac{1}{3}\left(r([\eta,\xi]_r,\varepsilon)-r(\xi\circ l_{\eta^{\#}},\varepsilon)\right)\\
		&=\frac{1}{3}\left(\langle \varepsilon,[\eta^{\#},\xi^{\#}]_{\mathfrak{m}}\rangle+\langle \xi,[\eta^{\#},\varepsilon^{\#}]_{\mathfrak{m}}\rangle\right)\\
		&=\frac{1}{3}\left(\langle \ad_{\xi^{\#}}^*\widetilde{\varepsilon},\eta^{\#}\rangle+\langle \ad_{\varepsilon^{\#}}^*\widetilde{\xi},\eta^{\#}\rangle\right)
	\end{align*}
In the same way we have
\[ r(\xi,\mathfrak{b}(\eta,\varepsilon))=-\frac{1}{3}\left(\langle \ad_{\varepsilon^{\#}}^*\widetilde{\xi},\eta^{\#}\rangle+\langle \ad_{\xi^{\#}}^*\widetilde{\varepsilon},\eta^{\#}\rangle\right).\]
Hence
\[ 	r(\mathfrak{b}(\eta,\xi),\varepsilon)+	r(\xi,\mathfrak{b}(\eta,\varepsilon))=0.\]
From the second point of Lemma \ref{Canonical connections} we get
\begin{align*}
	\mathfrak{b}(\eta,\xi)-\mathfrak{b}(\xi,\eta)&=\frac{2}{3}[\eta,\xi]_r-\xi\circ l_{\eta^{\#}}+\eta\circ l_{\xi^{\#}}\\
	&=[\eta,\xi]_r.
\end{align*}
\end{proof}

In general, there is no way to construct a covariant connection on the symplectic foliation arising from a contravariant connection. However, in our context, we will construct a covariant connection on the regular symplectic foliation induced by a an invaraiant contravariant connection.\\
Suppose that the pair $(G, H)$ is a reductive pair with the decomposition $\mathfrak{g} = \mathfrak{h} \oplus \mathfrak{m}$. Let $D$ be a $G$-invariant contravariant connection on $(G/H, \pi)$, and denote by $\mathfrak{b} : \mathfrak{m}^* \times \mathfrak{m}^* \to \mathfrak{m}^*$ its associated bilinear map. 

For any $u \in \mathfrak{m}$, denote by $u_r \in \mathrm{Im}(r_{\#})$ the component of $u$ in $\mathrm{Im}(r_{\#})$. For any $v \in \mathrm{Im}(r_{\#})$, denote by $\eta_v \in \mathfrak{m}^*$ the linear form defined by
$\prec \eta_v, u \succ = \omega_r(v, u_r)$.
\begin{theorem}
	The bilinear map $\mathfrak{b}^r : \mathrm{Im}(r_{\#}) \times \mathrm{Im}(r_{\#}) \to \mathrm{Im}(r_{\#})$ defined by
	\[
	\mathfrak{b}^r(u, v) = \mathfrak{b}(\eta_u, \eta_v)^{\#}
	\]
	is well-defined and $\mathrm{Ad}(H)$-invariant. Hence, it defines a $A_r$-invariant covariant connection $\nabla^r$ on the reductive pair $(A_r, H)$. Moreover, 
	\begin{enumerate}
		\item If $D$ is torsionless, then $\nabla^r$ is torsionless.
		\item If $D$ is a Poisson connection, then $\nabla^r$ is a symplectic connection.
		\item If $D$ is a Fedosov contravariant connection, then $\nabla^r$ is a Fedosov connection.
		\item If the curvature of $D$ vanishes, then  the curvature of $\nabla^r$ vanishes if and only, for all $u,v,w\in\mathrm{Im}(r_{\#})$, 
		\[ [[u,v]_{\mathfrak{h}},w]=0.\]
	\end{enumerate}
\end{theorem}
\begin{proof}
	Let $u,v,w\in\mathrm{Im}(r_{\#})$ and $a\in H$,
	\begin{align*}
	\prec \eta_{\Ad_a(u)}, v \succ &= \omega_r(\Ad_a(u), v_r)\\
	&=\omega_r(u, \Ad_{a^{-1}}(v_r))\\
	&=\prec \eta_{u}, \Ad_{a^{-1}}(v) \succ\\
	&= \prec \Ad_{a}^{*}\eta_{u}, v \succ.
	\end{align*} 
	Hence, 
	\begin{align*}
	\mathfrak{b}^r(\Ad_a(u), \Ad_a(v)) &= \mathfrak{b}(\eta_{\Ad_a(u)}, \eta_{\Ad_a(v)})^{\#}\\
	&= \mathfrak{b}(\Ad_{a}^{*}\eta_{u}, \Ad_{a}^{*}\eta_{v})^{\#}\\
	&= (\Ad_{a}^{*}\mathfrak{b}(\eta_{u},\eta_{v}))^{\#}\\
	&=\Ad_{a}(\mathfrak{b}(\eta_{u},\eta_{v})^{\#}) \\
	&= \Ad_{a}(\mathfrak{b}^r(\eta_{u},\eta_{v})).
	\end{align*}
	Thus $\mathfrak{b}^r$ is $A_r$-invariant.
	
	The others assertions are obvious.
\end{proof}

\bibliographystyle{amsplain}

\end{document}